\documentclass[twoside, final]{siamltex}
\usepackage{amsmath}
\usepackage{mathtools}
\usepackage{amssymb}
\usepackage{amsfonts}
\usepackage{amsxtra}
\usepackage{amstext}
\usepackage{amsbsy}
\usepackage{amscd}
\usepackage{graphicx}
\usepackage{float}
\usepackage{cite}
\usepackage{xcolor}
\usepackage{srcltx} 
\usepackage{marginnote,slashbox}
\usepackage{rotating}
\definecolor{darkblue}{rgb}{0.0,0.0,0.6}
\definecolor{darkgreen}{rgb}{0.0,0.6,0.0}
\usepackage[pdftex,colorlinks=true,urlcolor=darkblue,citecolor=darkblue,linkcolor=darkblue]{hyperref}
\usepackage[utf8]{inputenc}      
\usepackage{booktabs}            
\usepackage{url}                 
\usepackage[T1]{fontenc}         

\usepackage{calc}

\numberwithin{table}{section}    
\numberwithin{figure}{section}   
\numberwithin{equation}{section} 

\newcommand{\dx}{\textup{d}x}

\newcommand{\sign}{\operatorname{sign}}
\newcommand{\absop}{\operatorname{abs}}
\renewcommand{\supp}{\operatorname{supp}}
\newcommand{\dist}{\operatorname{dist}}
\newcommand{\R}{\mathbb{R}}
\newcommand{\N}{\mathbb{N}}

\newcommand{\BB}{\mathcal{B}}
\newcommand{\II}{\mathcal{I}}
\newcommand{\JJ}{\mathcal{J}}
\renewcommand{\AA}{\mathcal{A}}
\newcommand{\UU}{\mathcal{U}}
\newcommand{\KK}{\mathcal{K}}
\newcommand{\TT}{\mathcal{T}}
\newcommand{\Uad}{U_{\mathrm{ad}}}

\newcommand{\weak}{\rightharpoonup}
\newcommand{\vdual}[2]{\left\langle #1 ,\, #2 \right\rangle}
\newcommand{\Bigdual}[2]{\Big\langle #1 ,\, #2 \Big\rangle}
\newcommand{\dual}[2]{\langle #1 ,\, #2 \rangle}

\newcommand{\embed}{\hookrightarrow}
\newcommand{\OO}{\mathcal{O}}
\newcommand{\ddp}[2]{\frac{\partial #1}{\partial #2}}

\newcommand{\clos}[1]{\overline{#1}}
\newcommand{\interior}{\operatorname{int}}
\newcommand{\capa}{\operatorname{cap}}

\newtheorem{assumption}[theorem]{Assumption}
\newtheorem{remark}[theorem]{Remark}
\newenvironment{proofof}[1]{\emph{Proof of #1.}}{\endproof}

\begin{document}
\title{Strong stationarity conditions for a class of optimization problems governed by variational inequalities of the 2nd kind}
\date{\today}
\author{J.~C.~De Los Reyes\footnotemark[3] \and C.~Meyer\footnotemark[4]}
\renewcommand{\thefootnote}{\fnsymbol{footnote}}
\footnotetext[3]{Research Center on Mathematical Modelling (MODEMAT), Escuela Politécnica Nacional, Quito-Ecuador} 
\footnotetext[4]{Faculty of Mathematics, Technische Universit\"at Dortmund, Dortmund-Germany.}
\renewcommand{\thefootnote}{\arabic{footnote}}

\maketitle
\begin{abstract}
We investigate optimality conditions for optimization problems constrained by a class of variational inequalities of the second kind. Based on a nonsmooth primal-dual reformulation of the governing inequality, the differentiability of the solution map is studied. Directional differentiability is proved both for finite-dimensional and function space problems, under suitable assumptions on the active set. A characterization of B- and strong stationary optimal solutions is obtained thereafter. Finally, based on the obtained first-order information, a trust-region algorithm is proposed for the solution of the optimization problems. 
\end{abstract}

\section{Introduction}
Optimization problems with variational inequality constraints have been intensively investigated in the last years with many important applications in focus. Problems in contact mechanics, phase separation or elastoplasticity are some of the most relevant application examples. Special analytical and numerical techniques have been developed for characterizing and finding optima of such problems, mainly in the finite-dimensional case (see \cite{LuoPangRalph} and references therein).

In the function space framework much of the work has been devoted to optimization problems constrained by variational inequalities of the first kind:
\begin{subequations}
\begin{align}
\min ~& j(y,u)\\
\text{subject to: }& (Ay,v-y) \geq (u, v-y), \text{ for all } v \in K,
\end{align}
\end{subequations}
where $A :V \mapsto V^*$ is an elliptic operator and $K \subset V$ is closed convex set. Such obstacle type structure has allowed to develop an analytical machinery for such kind of problems. In addition, different type of stationarity concepts have been investigated in that framework (\emph{C-}, \emph{B-}, \emph{M-} and strong stationary points). The utilized proof techniques include regularization approaches as well as differentiability properties (directional, conic) of the solution map or elements of set valued analysis
(see e.g. \cite{Mignot1976, MignotPuel1984, Barbu1993, Bergounioux1998, hintermuller2009mathematical, OutrataJarusekStara2011, KunischWachsmuth2012a, KunischWachsmuth2012b, HerzogMeyerWachsmuth, SchielaWachsmuth2011, HMS13}).

For problems involving variational inequalities of the second kind: 
\begin{subequations}
\begin{align}
\min ~& j(y,u)\\
\text{subject to: }& (Ay,v-y) +\varphi(v)- \varphi(y) \geq (u, v-y),  \text{ for all } v \in V,
\end{align}
\end{subequations}
with $\varphi$ continuous and convex, only weak results have been obtained in the past, due to the very general structure (see e.g. \cite{Barbu1993,Bergounioux1998,BoCa2,outrata2000}). In \cite{Delosreyes2009} a special class of problems were investigated, where a richer structure of the nondifferentiability was exploited. Nonsmooth terms of the type $\varphi (y)=\int_S |B y|~ds$ were considered there and, by using a tailored regularization approach, a more detailed optimality system was obtained. The results were then extended to problems in fluid mechanics \cite{dlRe2010}, image processing \cite{dlReSchoen2013} and elastoplasticity \cite{dlRHM13}. Thanks to the availability of primal and dual formulations in elastoplasticity, the kind of optimality systems obtained in \cite{Delosreyes2009} were proved to be equivalent to C-stationary optimality systems in optimization problems constrained by variational inequalities of the first kind, see \cite{dlRHM13}.

In this paper we aim to characterize further stationary points by investigating differentiability properties of the solution map. In that spirit B- and strong stationarity conditions are in focus. To avoid problems related to the regularity of the variables, we start by considering the finite-dimensional case. A reformulation of the variational inequality as a nonsmooth system of primal dual equations enables us to take difference quotients and prove directional differentiability of the finite-dimensional solution operator. 

The technique is then extended to the function space setting. Since in this context the regularity of the functions as well as the structure of the active set play a crucial role, special functional analysis and measure theoretical methods have to be considered. As a preparatory step, the Lipschitz continuity of the solution operator from $L^p(\Omega) \to L^\infty(\Omega)$ is proved by using Stampacchia's technique. The directional differentiability of the solution map is then proved by assuming that the active set has a special structure, namely that it consists of the union of a regular subdomain of positive measure and a set of zero capacity (see Assumption \ref{assu:active} below). With the directional differentiability at hand, the characterization of B-stationarity points is carried out thereafter. The theoretical part of the paper ends with the derivation of strong stationarity conditions by 
an adaptation of the method of proof introduced by \cite{MignotPuel1984} for optimal control of the obstacle problem.

In the last part of the paper the first order information related to the directional derivative is utilized within a trust-region algorithm for the solution of the VI-constrai\-ned optimization problem. The computed derivative information is treated as an inexact descent direction, which is inserted into the trust-region framework to get robust iterates. The performance of the resulting algorithm is tested on a representative test problem, showing the suitability of the approach.

\section{Differentiability for a finite dimensional VI of second kind}
\label{sec:finite}

We start by considering the following prototypical VI in $\R^n$:
\begin{equation}\label{eq:vi2}
 \dual{A y}{v-y} + |v|_1 - |y|_1 \geq \dual{u}{v-y} \quad \forall \, v\in \R^n.
\end{equation}
Throughout this section $\dual{.}{.} = \dual{.}{.}_{\R^n}$ denotes the Euclidean scalar product. 
Moreover, $A\in \R^{n\times n}$ is positive definite and $|v|_1 = \sum_{i=1}^n |v_i|$. 
Existence and uniqueness for \eqref{eq:vi2} for arbitrary right hand sides $u\in \R^n$ follows by classical arguments 
due to the maximal monotonicity of $A + \partial |\,.\,|_1$.

\begin{definition}
 We denote the solution mapping associated to \eqref{eq:vi2} by $S: \R^n \ni u \mapsto y \in \R^n$.
\end{definition}

Next let us introduce a dual (slack) variable $q \in \R^n$ by $q := u - A y$.
If we test \eqref{eq:vi2} with $v_i = 0$, $v_i = 2 y_i$, and $v_i = q_i + y_i$ and $v_j = y_j$ for all $j\neq i$, 
then the following complementarity-like equivalent problem is obtained:
\begin{equation*}
 \left\{\;
 \begin{aligned}
  A y + q &= u &&\\
  q_i y_i &= |y_i|, \quad && i = 1, 2, ..., n \\ 
  |q_i| &\leq 1, \quad && i = 1, 2, ..., n, 
 \end{aligned}
 \right.
\end{equation*}
which can be reformulated as the following system of nonsmooth equations
\begin{equation}\label{eq:complsys}
 \left\{\;
 \begin{aligned}
  A y + q &= u &&\\
  q_i y_i &= |y_i|, \quad && i = 1, 2, ..., n \\ 
  \max\{|q_i|,1\} &= 1, \quad && i = 1, 2, ..., n.
 \end{aligned}
 \right.
\end{equation}
In order to derive a directional derivative for $S$, consider a perturbed version of \eqref{eq:vi2}, given by
\begin{equation}\label{eq:pert}
\begin{aligned}
 A y^t + q^t &= u + t\,h\\
 q^t_i y^t_i &= |y^t_i|, \quad i = 1, 2, ..., n \\ 
 \max\{|q^t_i|,1\} &= 1, \quad i = 1, 2, ..., n, 
\end{aligned}
\end{equation}
which leads to the following nonsmooth system for the difference quotient:
\begin{equation}\label{eq:diffquot}
\begin{aligned}
 A\, \frac{y^t - y}{t} + \frac{q^t - q}{t} &= h\\
 \frac{q^t_i y^t_i - q_i y_i -(|y^t_i| - |y_i|)}{t} &= 0, \quad i = 1, 2, ..., n \\ 
 \frac{\max\{|q^t_i|,1\} - \max\{|q_i|,1\}}{t} &= 0, \quad i = 1, 2, ..., n.
\end{aligned}
\end{equation}
In the sequel, we will pass to the limit in \eqref{eq:diffquot} to obtain the relations determining the directional derivative of $S$. 
For this purpose we test the VI associated with \eqref{eq:pert}, given by
\begin{equation}\label{eq:vi2pert}
 \dual{A y^t}{v-y^t} + |v|_1 - |y^t|_1 \geq \dual{u + t h}{v-y^t} \quad \forall \, v\in \R^n,
\end{equation}
with $v=y$. If we test \eqref{eq:vi2} with $v=y^t$ and add both inequalities, we arrive at
\begin{equation*}
 \lambda_{\min}(A) \Big|\frac{y^t - y}{t}\Big|^2 
 \leq \vdual{\frac{y^t - y}{t}}{A\, \frac{y^t - y}{t}} \leq \vdual{h}{\frac{y^t - y}{t}},
\end{equation*}
where $|\,.\,| = |\,.\,|_{\R^n}$ denotes the euclidian norm  and $\lambda_{\min}(A) > 0$ is the smallest eigenvalue of $A$. Thus
\begin{equation*}
 \Big|\frac{y^t - y}{t}\Big| \leq \frac{1}{\lambda_{\min}(A)}\, |h| < \infty,
\end{equation*} 
and so there exists a converging subsequence, w.l.o.g.\ $\left\{\frac{y^t - y}{t}\right\}_{t>0}$ itself, such that
\begin{equation}\label{eq:convy}
 \frac{y^t - y}{t} \stackrel{t \searrow 0}{\longrightarrow} \eta.
\end{equation}
In Theorem \ref{thm:rablfinite} below we will see that the limit $\eta$ is unique so that the whole sequence $\{(y^t - y)/t\}$ converges. 
This justifies to assume the convergence of the whole sequence right from the beginning. By definition of $q$ we have
\begin{equation}\label{eq:convq}
 \frac{q^t - q}{t} = h  - A \, \frac{y^t - y}{t} \stackrel{t \searrow 0}{\longrightarrow} h - A \eta =: \lambda,
\end{equation}
which in particular implies $q^t \to q$.

\begin{lemma}\label{lem:rablabs}
 For all $i=1, 2, ..., n$ there holds
 \begin{equation}
  \frac{q^t_i y^t_i - q_i y_i -(|y^t_i| - |y_i|)}{t} \stackrel{t \searrow 0}{\longrightarrow} 
  \lambda_i y_i + q_i \eta_i - f_i(\eta_i)
 \end{equation}
 with 
 \begin{equation*}
  f_i: \R\to \R, \quad f_i(x) := 
  \begin{cases}
   \sign(y_i)x, & y_i \neq 0,\\
   |x|, & y_i = 0.
  \end{cases}
 \end{equation*}
\end{lemma}

\begin{proof}
 We start by estimating
 \begin{equation*}
 \begin{aligned}
  &\Big|\frac{q^t_i y^t_i - q_i y_i -(|y^t_i| - |y_i|)}{t} - \lambda_i y_i - q_i \eta_i + f_i(\eta_i)\Big|\\
  &\quad \leq 
  \begin{aligned}[t]
   \Big|\Big(\frac{q^t - q}{t} - \lambda_i\Big) y_i\Big| 
   &+ \Big|q^t_i\,\frac{y^t - y}{t} - q_i \eta_i \Big|\\
   &+ \Big|\frac{|y^t_i| - |y_i + t\eta_i|}{t}\Big| + \Big| \frac{|y_i + t\eta_i| - |y_i|}{t} - f_i(\eta_i)\Big|.
  \end{aligned}   
 \end{aligned}
 \end{equation*}
 Because of \eqref{eq:convy}, \eqref{eq:convq}, and $q^t \to q$, the first two terms converge to zero. 
 Moreover, due to \eqref{eq:convy}, it holds
 \begin{equation*}
  \frac{y^t_i - y_i}{t} = \eta_i + \OO(t)
 \end{equation*}
 and thus the strong triangle inequality gives
 \begin{equation*}
  \Big|\frac{|y^t_i| - |y_i + t\eta_i|}{t}\Big|
  \leq \Big|\frac{y^t_i - y_i - t\eta_i}{t}\Big| = \OO(t) \to 0.
 \end{equation*}
 Moreover, $f_i$ is just the directional derivative of $|\,.\,|: \R \to \R$ at $y_i$ so that
 \begin{equation*}
  \frac{|y_i + t\eta_i| - |y_i|}{t} - f_i(\eta_i) \stackrel{t\searrow 0}{\longrightarrow} 0.
 \end{equation*}
 Altogether this implies the assertion.
\end{proof}

\begin{lemma}\label{lem:rablmax}
 The function $g: \R \to \R$, $g(x) = \max\{|x|, 1\}$ is directionally differentiable with 
 \begin{equation}\label{eq:rablmax}
  g'(x;h) = 
  \begin{cases}
   0, & |x| < 1\\
   \sign(x)h, & |x| > 1\\
   \max\{0, x h\}, & |x| = 1.
  \end{cases}
 \end{equation}
\end{lemma}

\begin{proof}
Let us define 
 \begin{equation*}
  B_t := \frac{\max\{|x+ th|, 1\} - \max\{|x|,1\}}{t}.
 \end{equation*}
 If $|x| < 1$, then $|x+th| < 1$ and thus $B_t = 0$ for sufficiently small $t>0$. 
 If $|x| > 1$, then $|x+th| > 1$ for sufficiently small $t>0$ and thus 
 \begin{equation*}
  B_t = \frac{|x + th| - |x|}{t} \stackrel{t\searrow 0}{\longrightarrow} \sign(x)h.
 \end{equation*}  
 If $|x| = 1$, then a simple distinction of cases shows that for $t> 0$ sufficiently small
 \begin{equation*}
  B_t  
  = \begin{cases}
   0, & h x < 0\\
   \sign(x) h, & h x \geq 0.
  \end{cases}
 \end{equation*}
 The right hand side is equivalent to $\max\{0,x h\}$ as we will see in the following.
 This is clear for $h x < 0$. If $h x = 0$, then $h = 0$ (since $|x| = 1$) and thus $\max\{0,x h\} = 0 = \sign(x) h$. 
 If $h x > 0$, then 
 \begin{equation*}
  \max\{0,x h\} = x h = |x| |h| = |h| = \sign(h) h = \sign(x) h.
 \end{equation*}
 All in all we have proven the assertion.
\end{proof}

\begin{lemma}\label{lem:rablmax2}
 For every $i=1, 2, ..., n$ there holds
 \begin{equation*}
  \frac{\max\{|q^t_i|,1\} - \max\{|q_i|,1\}}{t} \stackrel{t\searrow 0}{\longrightarrow} g'(q_i;\lambda_i)
 \end{equation*}
 with $g_i'$ as defined in \eqref{eq:rablmax}.
\end{lemma}

\begin{proof}
 The proof is similar to the one of Lemma \ref{lem:rablabs}. We estimate
 \begin{multline*}
  \Big|\frac{\max\{|q^t_i|,1\} - \max\{|q_i|,1\}}{t} - g'(q_i;\lambda_i)\Big| \\
  \leq \Big|\frac{\max\{|q^t_i|,1\} - \max\{|q_i + t \lambda_i|,1\}}{t}\Big| \\
  + \Big|\frac{\max\{|q_i + t \lambda_i|,1\} - \max\{|q_i|,1\}}{t} - g'(q_i;\lambda_i)\Big|.
 \end{multline*}
 The second addend tends to zero due to Lemma \ref{lem:rablmax}. 
 Moreover, by \eqref{eq:convq} we have $(q_i^t - q_i)/t = \lambda_i + \OO(t)$ and thus we find for the first addend, 
 by employing the Lipschitz continuity of $x\mapsto \max\{1,x\}$ and again the strong triangle inequality,
 \begin{equation*}
  \Big|\frac{\max\{|q^t_i|,1\} - \max\{|q_i + t \lambda_i|,1\}}{t}\Big| 
  \leq \Big|\frac{|q^t_i| - |q_i + t \lambda_i|}{t}\Big| 
  \leq \Big| \frac{q^t_i - q_i - t \lambda_i}{t} \Big|
  \leq \OO(t) \to 0,
 \end{equation*}
 which gives the assertion.
\end{proof}

In view of \eqref{eq:convy}, \eqref{eq:convq}, and lemmas \ref{lem:rablabs} and \ref{lem:rablmax2}, we can pass to the limit as
$t\searrow 0$ in \eqref{eq:diffquot} and obtain in this way:
\begin{subequations}\label{eq:rablcomplsys}
\begin{align}
 A \eta + \lambda &= h\\
 \lambda_i y_i + q_i \eta_i &= 
 \begin{cases}
  \sign(y_i)\eta_i, & y_i \neq 0,\\
  |\eta_i|, & y_i = 0,
 \end{cases}
 \qquad i = 1, 2, ..., n \label{eq:rablb}\\
 \max\{0,q_i\lambda_i\} &= 0 \quad \text{for all } i \in \{1, ..., n\} \text{ with } |q_i| = 1. \label{eq:rablc}
\end{align}
\end{subequations}
(Note that the case $|q_i|> 1$ is obsolete.) The system \eqref{eq:rablcomplsys} will lead to a VI satisfied by the limit $\eta$. 
To see this, we have to reformulate \eqref{eq:rablcomplsys} in the following way:

\begin{lemma}
 The system \eqref{eq:rablcomplsys} is equivalent to
 \begin{subequations}\label{eq:rablcomplsys2}
 \begin{align}
  A \eta + \lambda &= h \label{eq:rabl2a}\\
  \lambda_i &= 0 \quad \text{for all } i \in \{1, ..., n\} \text{ with } y_i \neq 0 \label{eq:rabl2b}\\
  \eta_i &= 0 \quad \text{for all } i \in \{1, ..., n\} \text{ with } |q_i| < 1 \label{eq:rabl2c}\\
  \eta_i q_i &\geq 0 \quad \text{for all } i \in \{1, ..., n\} \text{ with } y_i = 0,\; |q_i| = 1 \label{eq:rabl2d}\\
  \lambda_i q_i &\leq 0 \quad \text{for all } i \in \{1, ..., n\} \text{ with } y_i = 0,\; |q_i| = 1. \label{eq:rabl2e}
 \end{align}
 \end{subequations}
\end{lemma}

\begin{proof}
 $\eqref{eq:rablcomplsys} \Rightarrow \eqref{eq:rablcomplsys2}$:\\
 It is evident that 
 \begin{equation}\label{eq:maxequiv}
  \max\{0,q_i\lambda_i\} = 0 \text{ if } |q_i| = 1 
  \quad \Longleftrightarrow \quad 
  q_i\lambda_i \leq 0 \text{ if } |q_i| = 1,
 \end{equation}
 which implies \eqref{eq:rabl2e}. Next, let $i\in \{1, ..., n\}$ such that $y_i \neq 0$. 
 Then
 \begin{equation*}
  q_i = \frac{y_i}{|y_i|} = \sign(y_i),
 \end{equation*}
 and hence \eqref{eq:rablb} yields $\lambda_i y_i = 0$, which in turn gives \eqref{eq:rabl2b} due to $y_i \neq 0$.
 Now take $i\in \{1, ..., n\}$ with $|q_i| < 1$ arbitrary. Then we have $y_i = 0$, and hence \eqref{eq:rablb} implies
 $q_i \eta_i = |\eta_i|$. Because of $|q_i| < 1$ this results in \eqref{eq:rabl2c}. 
 To show \eqref{eq:rabl2d}, let $i\in \{1, ..., n\}$ with $y_i = 0$ and $|q_i| = 1$ 
 be arbitrary. Then \eqref{eq:rablb} gives $q_i \eta_i = |\eta_i| \geq 0$.

 $\eqref{eq:rablcomplsys2} \Rightarrow \eqref{eq:rablcomplsys}$:\\
 Due to \eqref{eq:rabl2b} and \eqref{eq:rabl2e} we have $\lambda_i q_i \leq 0$ whenever $|q_i| = 1$, 
 which, in view of \eqref{eq:maxequiv}, implies \eqref{eq:rablc}. 
 Because of \eqref{eq:rabl2b}, we have
 \begin{equation}\label{eq:rablequiv1}
  \lambda_i y_i + \eta_i q_i = \eta_i q_i \quad \forall \, i = 1, ..., n.
 \end{equation}
 Now, if $y_i \neq 0$, then $q_i = \sign(y_i)$ and thus $\eta_i q_i = \sign(y_i) \eta_i$.
 If $y_i = 0$ and $|q_i| < 1$, then, by \eqref{eq:rabl2c}, we obtain  $\eta_i q_i = 0 = |\eta_i|$.
 If finally $y_i = 0$ and $|q_i| = 1$, then \eqref{eq:rabl2d} implies $\eta_i q_i  = |\eta_i||q_i| = |\eta_i|$.
 In summary \eqref{eq:rablb} is verified, which yields the assertion.
\end{proof}

System \eqref{eq:rablcomplsys2} is not yet complete, since there is still one relation missing to derive the VI fulfilled 
by $\eta$. The missing part is stated in the following lemma.

\begin{lemma}\label{lem:etalam}
 There holds
 \begin{equation*}
  \eta_i \lambda_i = 0 \quad \text{for all } i \in \{1, ..., n\} \text{ with } y_i = 0,\; |q_i| = 1.
 \end{equation*}
\end{lemma}

\begin{proof}
 Let $i\in \{1, ..., n\}$ with $y_i = 0$ and $|q_i| = 1$ be arbitrary. W.l.o.g.\ we assume that $q_i = 1$. 
 The case $q_i = -1$ can be discussed analogously. If $\eta_i = 0$, the assertion is trivially fulfilled. So let $\eta_i \neq 0$. 
 By \eqref{eq:rabl2d} and $q_i = 1$ we then have $\eta_i > 0$. Due to \eqref{eq:convy} this implies
 \begin{equation*}
  \frac{y_i^t - y_i}{t} > 0 \quad \text{for } t > 0 \text{ sufficiently small}
 \end{equation*}
 and thus, due to $y_i = 0$, 
 \begin{equation*}
  y_i^t > 0 \quad \text{for } t > 0 \text{ sufficiently small.}
 \end{equation*}
 Consequently, $q_i^t = \sign(y^t_i) = 1$ for $t>0$ sufficiently small and hence, since $q_i = 1$ by assumption, 
 \begin{equation*}
  \lambda_i = \lim_{t\searrow 0} \frac{q^t_i - q_i}{t} = 0,
 \end{equation*}
 which gives the assertion.
\end{proof}

Now we have everything at hand to prove the main result of this section, i.e., the directional differentiability of $S: u\mapsto y$.

\begin{theorem}\label{thm:rablfinite}
 The solution mapping $S$ of \eqref{eq:vi2} is directionally differentiable at every point $u\in \R^n$ and 
 the directional derivative $\eta = S'(u;h)$ in direction $h\in \R^n$ solves the following \emph{VI of first kind}:
 \begin{equation}\label{eq:vi1}
  \eta \in K(y), \quad \dual{A \eta}{v - \eta} \geq \dual{h}{v-\eta} \quad \forall\, v\in K(y)
 \end{equation}
 where $K(y)$ is the convex cone defined by
 \begin{equation}\label{eq:conefinite}
  K(y) := \{ v\in \R^n: v_i = 0 \text{ if } |q_i| < 1,\; v_i q_i \geq 0 \text{ if } y_i = 0, \, |q_i| = 1 \}.
 \end{equation}
\end{theorem}

\begin{proof}
 Define the biactive set by 
 \begin{equation*}
  \BB := \{i \in \{1, ..., n\} : y_i = 0,\; |q_i| = 1\}.
 \end{equation*}
 Fist we show that the limit $\eta$ solves \eqref{eq:vi1}. We already know that $\eta$ satisfies \eqref{eq:rablcomplsys2} 
 and in addition $\eta_i \lambda_i = 0$ if $y_i = 0$ and $|q_i| = 1$.
 Thus \eqref{eq:rabl2c} and \eqref{eq:rabl2d} imply $\eta \in K(y)$, i.e., feasibility of $\eta$. 
 Now let $v\in K(y)$ be arbitrary. Then \eqref{eq:rabl2b}, $v\in K(y)$, and \eqref{eq:rabl2e} yield
 \begin{equation}\label{eq:lambdasign}
  \dual{\lambda}{v} = \sum_{i\in\BB} \lambda_i v_i = \sum_{i\in\BB} \lambda_i \underbrace{q_i q_i}_{=1} v_i \leq 0.
 \end{equation}
 Similarly, we infer from \eqref{eq:rabl2b}, $\eta \in K(y)$, and  Lemma \ref{lem:etalam} that
 \begin{equation*}
  \dual{\lambda}{\eta} = \sum_{i\in\BB} \lambda_i v_i = 0.
 \end{equation*}  
 Therefore, if we multiply \eqref{eq:rabl2a} with $v-\eta$, then we arrive at
 \begin{equation*}
 \begin{aligned}
  \dual{h}{v-\eta} = \dual{A\eta}{v-\eta} + \dual{\lambda}{v} - \dual{\lambda}{\eta} \leq \dual{A\eta}{v-\eta},
 \end{aligned}
 \end{equation*}
 so that the limit $\eta$ indeed solves \eqref{eq:vi1}.
 
 Since $A$ is positive definite and $K(y)$ is convex and closed, the operator $A + \partial I_{K(y)}(.): \R^n \to 2^{\R^n}$ 
 is maximal monotone, where $I_{K(y)}$ denotes the indicator function of the set $K(y)$. Thus there is a unique solution of \eqref{eq:vi1}.
 Since every accumulation point $\eta$ of the difference quotient $(y^t-y)/t$ solves \eqref{eq:vi1}, the limit is thus unique and 
 consequently a well-known argument gives the convergence of the whole sequence.
\end{proof}

\begin{corollary}
 Let the biactive set have zero cardinality, i.e.\ $y_i = 0$ implies $|q_i| < 1$. 
 Then $S$ is G\^ateaux-differentiable, i.e.\ $S'(u;h)$ is linear and continuous w.r.t.\ $h$, and $\eta = S'(u)h$ is given 
 by the unique solution of the following linear system:
 \begin{align}
  \eta_i &= 0 \quad \text{for all } i\in\{1, ..., n\} \text{ with } y_i = 0 \label{eq:etanull}\\ 
  \sum_{j: y_j \not = 0} A_{ij} \eta_j &= h_i \quad \text{for all } i\in\{1, ..., n\} \text{ with } y_i \neq 0.\label{eq:rableq}
 \end{align}
\end{corollary}

\begin{proof}
 If the biactive set has zero cardinality, then \eqref{eq:rabl2c} implies \eqref{eq:etanull}. 
 Moreover, \eqref{eq:rabl2b} immediately yields \eqref{eq:rableq}. Since $A$ is positive definite, the same holds for 
 $A_\II := (A_{ij})_{i,j\in\II}$ with $\II := \{ i \in \{1, ..., n\}: y_i \neq 0\}$. 
 Thus $A_\II$ is invertible and $\eta_\AA = A_\II^{-1} h_\II$. Together with \eqref{eq:etanull}, 
 i.e.\ $\eta_{\{1, ..., n\}\setminus\II} = 0$, this implies that $\eta$ is uniquely determined by \eqref{eq:etanull} and 
 \eqref{eq:rableq}. Moreover, due to the invertibility of $A_\II$, $\eta$ depends continuously on $h$ as claimed.
\end{proof}

\section{Weak differentiability for a VI of second kind in function space}

Next we extend the result of the preceding section to a VI of second kind in function space.
For this purpose, let $\Omega \subset \mathbb R^d$, $d\geq 1$, be a bounded domain with regular boundary satisfying the cone condition.
We consider the following prototypical VI of second kind:
\begin{equation}\tag{VI2}\label{eq:vi}
 \dual{A y}{v-y} + \int_\Omega |v|\,dx - \int_\Omega |y|\,dx \geq \dual{u}{v-y}
 \quad \forall \, v\in V,
\end{equation}
where we abbreviated $V := H^1_0(\Omega)$. From now on $\dual{.}{.}$ denotes the dual pairing in $V$.
Furthermore $A: V \to V^*$ stands for the following linear second-order elliptic differential operator:
\begin{equation}\label{eq:diffop}
 A y = \sum_{i=1}^d \Big(\sum_{j=1}^d \ddp{}{x_i} a_{ij} \ddp{y}{x_j} + b_i\,\ddp{y}{x_i}\Big) + \gamma\,y,
\end{equation}
where $a_{ij}, b_i, \gamma \in L^\infty(\Omega)$, $i,j=1, .., d$, are such that $A$ is coercive, i.e.
\begin{equation}\label{eq:coer}
 \dual{A y}{y} \geq \alpha \, \|y\|_{V}^2,
\end{equation}
with a constant $\alpha > 0$. In addition, we require 
\begin{equation}\label{eq:gammasign}
 \gamma \geq 0.
\end{equation}
Moreover, $u \in V^*$ is given a inhomogeneity.

The plan of this section is as follows. First we state some well known results for \eqref{eq:vi} concerning existence, uniqueness, 
and an equivalent reformulation by means of a complementarity-like system. Then we introduce a perturbed problem, similar to \eqref{eq:pert}, 
and derive several auxiliary results for the associated difference quotients and their (weak) limits. In order to show an infinite 
dimensional analogon to \eqref{eq:rabl2b}, we unfortunately need to assume some properties of the active set, 
see Assumption \ref{assu:active} below. Based on this assumption, we can derive a weak directional differentiability result, similar 
to Theorem \ref{thm:rablfinite} (see Theorem \ref{thm:ablvi} below).

\begin{lemma}\label{lem:lipschitz}
 For every $u\in V^*$ there exists a unique solution $y\in V$ of \eqref{eq:vi}, which we denote by $y = S(u)$. 
 The associated solution operator $S: V^* \to V$ is globally Lipschitz continuous, i.e., there exists a constant $L > 0$ 
 such that 
 \begin{equation}
  \|S(u_1) - S(u_2)\|_V \leq L \, \|u_1 - u_2\|_{V^*} \quad \forall \, u_1, u_2 \in V^*.
 \end{equation}
\end{lemma}

\begin{proof}
 Existence and uniqueness for \eqref{eq:vi} follows by standard arguments from the maximal monotonicity of $A + \partial \|.\|_{L^1(\Omega)}$, 
 see for instance \cite{Barbu1993}.
 To prove the Lipschitz continuity we test the VI for $y_1 = S(u_1)$ with $y_2 = S(u_2)$ and vice versa and add the arising 
 inequalities to obtain
 \begin{equation*}
  \dual{A(y_1 - y_2)}{y_1 - y_2} \leq \dual{u_1 - u_2}{y_1 - y_2}.
 \end{equation*}
 The coercivity of $A$ then yields the result.
\end{proof}

\begin{remark}
Sometimes we will use $S$ with different domains and ranges, which may be inferred from the context.
\end{remark}
By standard arguments based on Fenchel duality or the Hahn-Banach theorem, the VI in \eqref{eq:vi} can be rewritten in terms of 
a complementarity-like system, see e.g. \cite{Delosreyes2009}. In this way the following result is obtained:

\begin{lemma}\label{lem:slack}
 For every $u\in V^*$ there exists a unique function $q\in L^2(\Omega)$ such that the unique solution $y \in V$ of \eqref{eq:vi} fulfills
 the following complementarity-like system:
 \begin{subequations}
 \begin{gather}
  \dual{A y}{v} + \int_\Omega q\, v\,dx = \dual{u}{v} \quad \forall \, v\in V \label{eq:qdef}\\
  q(x) y(x) = |y(x)|,\quad |q(x)| \leq 1 \quad \text{a.e.\ in } \Omega.\label{eq:slacklike}
 \end{gather}
 \end{subequations}  
 The function $q$ is called slack function in all what follows, and we will refer to \eqref{eq:slacklike} as slackness condition 
 in the sequel.
\end{lemma}

Next let $h\in V^*$ be arbitrary and $\{t_n\} \subset \R^+$ be an arbitrary sequence of positive numbers tending to 0. 
We denote the solutions to the VI associated to $u+t_n h$ by $y_n$, i.e., 
\begin{equation}\label{eq:ynvi}
 \dual{A y_n}{v-y_n} + \int_\Omega |v|\,dx - \int_\Omega |y_n|\,dx \geq \dual{u+t_n h}{v-y_n}
 \quad \forall \, v\in V.
\end{equation}
The associated slack function is analogously denoted by $q_n \in L^2(\Omega)$, i.e.
\begin{equation}\label{eq:qn}
\begin{gathered}
 \dual{A y_n}{v} + \int_\Omega q_n\, v\,dx = \dual{u + t_n h}{v} \quad \forall \, v\in V, \\
 q_n(x) y(x) = |y_n(x)|,\quad |q_n(x)| \leq 1 \quad \text{a.e.\ in } \Omega.
\end{gathered}
\end{equation}
By Lemma \ref{lem:lipschitz} it holds
\begin{equation*}
 \Big\| \frac{y_n - y}{t_n}\Big\|_V \leq \|h\|_{V^*}
\end{equation*}
and thus there is a weakly convergent subsequence, denoted the same, and a limit point $\eta \in V$ such that
\begin{equation}\label{eq:weakV}
 \frac{y_n - y}{t_n} \weak \eta \text{ in } V.
\end{equation}
This simplification of notation will be justified by the uniqueness of the weak limit $\eta$, which implies the 
weak convergence of the whole sequence by a well-known argument (see Theorem \ref{thm:ablvi} below).
For the slack functions we obtain
\begin{equation*}
 \int_\Omega \frac{q_n - q}{t_n} \, v \, dx = \dual{h}{v} - \Bigdual{A \,\frac{y_n - y}{t_n}}{v} 
 \to \dual{h - A\eta}{v} \quad \forall\, v\in V,
\end{equation*}
i.e., 
\begin{equation*}
 \frac{q_n - q}{t_n} \weak \lambda \text{ in } V^*,
\end{equation*}
with $\lambda = h - A\eta$. 
Note that it is in general not possible to show the boundedness of $(q_n - q)/t_n$ in any Lebesgue space so that 
one cannot expect $\lambda$ to be more regular.

Next consider the first equation in the slackness condition \eqref{eq:slacklike} for $y$ and $y_n$. By multiplying these equations 
with $1/t_n$ and an arbitrary $\varphi\in C^\infty_0(\Omega)$, integrating over $\Omega$, and taking the difference, we arrive at
\begin{equation}\label{eq:slackdiff}
 \int_\Omega \frac{q_n - q}{t_n}\, y_n\, \varphi\, dx + \int_\Omega \frac{y_n - y}{t_n}\, q \, \varphi\, dx 
 = \int_\Omega \frac{|y_n| - |y|}{t_n}\,\varphi \, dx, \quad \forall\, \varphi \in C^\infty_0(\Omega).
\end{equation}
In order to pass to the limit in this relation, we have to define the following sets:
\begin{definition}\label{def:sets}
 We define --up to sets of zero measure--
 \begin{equation}\label{eq:defsets}
 \begin{aligned}
  \AA &:= \{x\in\Omega : y(x) = 0\}, & \AA_s &:= \{x\in \Omega: |q(x)| < 1\}\\
  \II &:= \{x\in \Omega: y(x) \neq 0\}, & \BB &:= \{x\in \Omega: |q(x)| = 1, \; y(x) = 0\}\\
  \II^+ &:= \{x\in \Omega: y(x) > 0\}, & 
  \II^- &:= \{x\in \Omega: y(x) < 0\}\\
  \BB^+ &:= \{x\in \Omega: q(x) = 1, \; y(x) = 0\}, & 
  \BB^- &:= \{x\in \Omega: q(x) = -1, \; y(x) = 0\}.
 \end{aligned}
 \end{equation}
 The set $\AA$ is called \emph{active set}, while $\AA_s$ is the \emph{strongly active set}. 
 Moreover, we call $\II$ and $\BB$ \emph{inactive} and \emph{biactive set}, respectively.
\end{definition}

Note that 
\begin{equation*}
 \Omega = \AA \cup \II \quad \text{and} \quad \AA = \AA_s \cup \BB,
\end{equation*}
due to \eqref{eq:slacklike}.
The next lemma covers the directional differentiability of the $L^1$-norm. Its proof is straightforward and therefore postponed to 
Appendix \ref{sec:l1deriv}.

\begin{lemma}\label{lem:l1deriv}
 For every $\varphi \in C^\infty_0(\Omega)$ it holds
 \begin{equation*}
  \int_\Omega \frac{|y_n| - |y|}{t_n}\,\varphi \, dx \to \int_\Omega \absop'(y;\eta)\,\varphi\, dx
 \end{equation*}
 with
 \begin{equation*}
  \absop'(y;\eta) \in L^2(\Omega), \quad \absop'(y;\eta)(x) := 
  \begin{cases}
   \sign\big(y(x)\big) \eta(x), & y(x) \neq 0\\
   |\eta(x)|, & y(x) = 0.
  \end{cases}
 \end{equation*}
\end{lemma}

Together with Lemma \ref{lem:l1deriv}
the weak convergence of $(q_n-q)/t_n$ in $V^*$ and $(y_n - y)/t_n$ in $V$ 
and the strong convergence of $y_n$ to $y$ in $V$ allow to pass to the limit in \eqref{eq:slackdiff}, which results in
\begin{equation}\label{eq:slackabl}
 \dual{\lambda}{y\,\varphi} + \int_\Omega \eta\, q \, \varphi\, dx = \int_\Omega \absop'(y;\eta)\,\varphi\,dx 
 \quad \forall \, \varphi \in C^\infty_0(\Omega).
\end{equation}
Using this relation, we can prove the following result, which is just the infinite dimensional counterpart to \eqref{eq:rabl2c} and 
\eqref{eq:rabl2d}:

\begin{lemma}\label{lem:etafeas}
 There holds
 \begin{align}
  \eta(x) &= 0 \quad \text{a.e., where } |q(x)| < 1 \label{eq:etaAs}\\
  \eta(x)\, q(x) &\geq 0 \quad \text{a.e., where } |q(x)| = 1 \text{ and } y(x) = 0.\label{eq:etaB}
 \end{align}
\end{lemma}

\begin{proof}
 Let $\varphi \in C^\infty_0(\Omega)$ with $\varphi \geq 0$ a.e.\ in $\Omega$ be arbitrary. 
 The slackness condition \eqref{eq:slacklike} implies for all $n\in \N$ that
 \begin{equation*}
  \frac{q_n(x) - q(x)}{t_n}\, y(x) \leq 0 \quad \text{a.e.\ in }\Omega.
 \end{equation*}
 Therefore we have
 \begin{equation*}
  \dual{\lambda}{y\, \varphi} = \lim_{n\to \infty} \int_\Omega \frac{q_n - q}{t_n}\, y\, \varphi \, dx \leq 0,
 \end{equation*}
 and thus \eqref{eq:slackabl} yields
 \begin{equation*}
  \int_\Omega \eta\, q \, \varphi\, dx \geq \int_\Omega \absop'(y;\eta)\,\varphi\,dx 
 \quad \forall \, \varphi \in C^\infty_0(\Omega) \text{ with } \varphi \geq 0.
 \end{equation*}
 The fundamental lemma of the calculus of variations thus yields
 \begin{equation*}
  \eta(x)\, q(x) \geq \absop'(y;\eta)(x) \quad \text{a.e.\ in } \Omega,
 \end{equation*}
 which by definition of $\absop'(y;\eta)$ in turn gives
 \begin{equation*}
  \eta(x) \, q(x) \geq |\eta(x)| \quad \text{a.e. in } \AA.
 \end{equation*}
 Since $|q(x)| \leq 1$ a.e.\ in $\Omega$, this results in 
 \begin{equation}\label{eq:etaqae}
  \eta(x) \, q(x) = |\eta(x)| \quad \text{a.e. in } \AA.
 \end{equation}
 As the slackness conditions in \eqref{eq:slacklike} implies $\{x\in\Omega: |q(x)| < 1\} \subset \{x\in\Omega: y(x) = 0\}$, 
 the result follows immediately from \eqref{eq:etaqae}.
\end{proof}

\begin{lemma}\label{lem:etalambdanull}
 There holds $\dual{\lambda}{\eta} \geq 0$.
\end{lemma}

\begin{proof}
 By inserting the definition of the slack variable $q$ into \eqref{eq:vi} one obtains
 \begin{equation}
  \int_\Omega q(v-y)\,dx \leq \int_\Omega |v|\, dx - \int_\Omega |y|\, dx \quad \forall\, v\in V
 \end{equation}
 and an analogous inequality for $q_n$ and $y_n$. 
 Inserting $y_n \in V$ in this inequality and $y$ in the corresponding one for $q_n$ and $y_n$, adding both inequalities 
 and dividing by $t_n^2$ yields
 \begin{equation*}
  \int_\Omega \frac{q_n - q}{t_n}\,\frac{y_n - y}{t_n}\,dx \geq 0.
 \end{equation*}
 Since $A$ is elliptic and bounded, the mapping $V \ni w \mapsto \dual{A w}{w} \in \R$ is convex and continuous and thus 
 weakly lower semicontinuous. The equations for $q$ and $q_n$ and the weak convergence of $(y_n - y)/t_n$ in $V$ therefore imply
 \begin{equation*}
 \begin{aligned}
  0 &\leq \liminf_{n\to\infty} \int_\Omega \frac{q_n - q}{t_n}\,\frac{y_n - y}{t_n}\,dx \\
  &\leq \limsup_{n\to\infty} \int_\Omega \frac{q_n - q}{t_n}\,\frac{y_n - y}{t_n}\,dx \\
  &= \limsup_{n\to\infty} \Big(\Bigdual{h}{\frac{y_n - y}{t_n}} - \Bigdual{A\Big(\frac{y_n - y}{t_n}\Big)}{\frac{y_n - y}{t_n}}\Big)\\
  &\leq \lim_{n\to\infty} \Bigdual{h}{\frac{y_n - y}{t_n}} - \liminf_{n\to\infty} \Bigdual{A\Big(\frac{y_n - y}{t_n}\Big)}{\frac{y_n - y}{t_n}}\\
  &\leq \dual{h}{\eta} - \dual{A \eta}{\eta} = \dual{\lambda}{\eta}.
 \end{aligned}
 \end{equation*}  
\end{proof}

The most delicate issue, when transferring the finite dimensional findings of Section \ref{sec:finite} to the function space setting, 
is to verify the conditions \eqref{eq:rabl2a} and \eqref{eq:rabl2e} on $\lambda$. 
To do so, we first prove that $S$ is Lipschitz continuous in $L^\infty(\Omega)$, provided that the right hand sides in \eqref{eq:vi} are 
more regular. We employ the well-known technique of Stampacchia based on the following lemma, whose proof is presented in Appendix \ref{sec:stam} for convenience of the reader.

\begin{lemma}[Stampacchia]\label{lem:stam}
 For every function $w\in V$ and every $k\geq 0$, the function $w_k$ defined by
 \begin{equation}\label{eq:truncfunc}
  w_k(x) := 
  \begin{cases}
   w(x) - k, & w(x) \geq k\\
   0, & |w(x)| < k\\
   w(x) + k, & w(x) \leq -k
  \end{cases}
 \end{equation}  
 is an element of $V$. Furthermore, if there is a constant $\alpha > 0$ such that 
 \begin{equation}\label{eq:stamest}
  \alpha \|w_k\|_{H^1(\Omega)}^2 \leq \int_\Omega f\,w_k\,dx \quad \forall\, k \geq 0
 \end{equation}
 with a function $f\in L^p(\Omega)$, $p> \max\{d/2, 1\}$,
 then $w$ is essentially bounded and there exists a constant $c>0$ so that 
 \begin{equation}\label{eq:inftybound}
  \|w\|_{L^\infty(\Omega)} \leq c\,\|f\|_{L^p(\Omega)}.
 \end{equation}
\end{lemma}

\begin{lemma}\label{lem:inftylip}
 There exists a constant $K>0$ such that 
 \begin{equation*}
  \|S(u_1) - S(u_2)\|_{L^\infty(\Omega)} \leq K\,\|u_1 - u_2\|_{L^p(\Omega)}
 \end{equation*}
 for all $u_1, u_2 \in L^p(\Omega)$ with $p> \max\{d/2, 1\}$. Here we identified $u \in L^p(\Omega)$ with an element of $V^*$.
\end{lemma}

\begin{proof}
 We apply Lemma \ref{lem:stam} to $w:= y_1 - y_2$ with $y_i = S(u_i)$, $i=1,2$. 
 To this end we shall verify \eqref{eq:stamest} with $f = u_1 - u_2$.
 For this purpose we test the VI for $y_1$ with $y_1 - v$ and the one for $y_2$ with $y_2 + v$ and add the arising 
 inequalities to obtain:
 \begin{equation}\label{eq:diffineq}
  \dual{A(y_1 - y_2)}{v} + \int_\Omega\big(|y_1| + |y_2| - |y_1 - v| - |y_2 + v|\big) dx
  \leq \int_\Omega (u_1 - u_2)v \, dx \quad \forall\, v\in V.
 \end{equation}
 Next let $k \geq 0$ be arbitrary and define $w_k = (y_1 - y_2)_k$ as in \eqref{eq:truncfunc}. 
 In the following we will prove that  
 \begin{equation}\label{eq:wksign}
  I(x) := |y_1(x)| + |y_2(x)| - |y_1(x) - w_k(x)| - |y_2(x) + w_k(x)| \geq 0 \quad\text{a.e.\ in } \Omega,
 \end{equation}
 by a simple distinction of cases.

 \emph{1st case: $|y_1(x) - y_2(x)| < k$:}\\
 In this case we have $w_k(x) = 0$ and thus \eqref{eq:wksign} is trivially fulfilled with equality. 

 \emph{2nd case: $y_1(x) - y_2(x) \geq k$:}\\
 Now we obtain $w_k(x) = y_1(x) - y_2(x) - k$ and consequently
 \begin{equation*}
  I(x) = |y_1(x)| + |y_2(x)| - |y_2(x) + k| - |y_1(x) - k|.
 \end{equation*}
 If $y_1(x) \geq k$ and $y_2(x) \leq -k$, then
 \begin{equation*}
  I(x) =  |y_1(x)| + |y_2(x)| + y_2(x) + k - y_1(x) + k \geq 2k \geq 0.
 \end{equation*}
 If $y_1(x) \leq k$ and $y_2(x) \geq -k$, then
 \begin{equation*}
  I(x) =  |y_1(x)| + |y_2(x)| - y_2(x) - k + y_1(x) - k 
  \geq 2\big( y_1(x) - y_2(x) - k \big)\geq 0,
 \end{equation*}
 where we used $y_1(x) - y_2(x) \geq k$ for the last estimate.\\
 If $y_1(x) \geq k$ and $y_2(x) \geq -k$, then
 \begin{equation*}
  I(x) =  |y_1(x)| + |y_2(x)| - y_2(x) - y_1(x) \geq 0.
 \end{equation*}
 If finally $y_1(x) \leq k$ and $y_2(x) \leq -k$, then
 \begin{equation*}
  I(x) =  |y_1(x)| + |y_2(x)| + y_2(x) + y_1(x) \geq 0,
 \end{equation*}
 which gives the assertion of \eqref{eq:wksign} for this case.

 \emph{3rd case: $y_1(x) - y_2(x) \leq -k$:}\\
 In this case we get that $y_2(x) - y_1(x) \geq k$ and thus $I(x) = |y_1(x)| + |y_2(x)| - |y_2(x) - k| - |y_1(x) + k|$.
 Interchanging the roles of $y_1(x)$ and $y_2(x)$ and repeating the arguments for the second case immediately 
 yields \eqref{eq:wksign} in the third case. 
 
 Let us now define $\AA_k := \{x\in \Omega: |w(x)| \geq k\}$.
 From the first part of Lemma \ref{lem:stam} we get that $w_k \in V$ and so we are allowed to insert $w_k$ 
 as test function in \eqref{eq:diffineq}. Owing to the coercivity of $A$, the definition of $w_k$ in \eqref{eq:truncfunc}, 
 \eqref{eq:gammasign}, and \eqref{eq:wksign}, we then obtain
 \begin{equation*}
 \begin{aligned}
  \alpha \, \|w_k\|_{H^1(\Omega)}^2 
  &\leq \dual{A w_k}{w_k}\\
  &= \int_{\AA(k)} \Big[\sum_{i=1}^d \Big( \sum_{j=1}^d a_{ij} \ddp{w_k}{x_j}\,\ddp{w_k}{x_j} dx
  + b_i \,\ddp{w_k}{x_i}\,w_k + \gamma\,\big(|w|-k\big)^2\Big]\dx\\
  &\leq \int_\Omega \Big[\sum_{i=1}^d \Big( \sum_{j=1}^d a_{ij} \ddp{w}{x_j}\,\ddp{w_k}{x_j} dx
  + b_i \,\ddp{w}{x_i}\,w_k + \gamma\,w\,w_k\Big]\dx\\
  &\leq \dual{A w}{w_k}  = \dual{A(y_1 - y_2)}{w_k} \leq \int_\Omega (u_1 - u_2) w_k\,dx,
 \end{aligned}
 \end{equation*}
 which is \eqref{eq:stamest} with $f = u_1 - u_2$.
 Since $k\geq 0$ was aribitrary, all conditions of Lemma \ref{lem:stam} are satisfied so that it can be applied and gives the desired 
 result.
\end{proof}

\begin{remark}\label{rem:infty}
 Since $S(0) = 0$, it immediately follows from Lemma \ref{lem:inftylip} that $$\|S(u)\|_{L^\infty(\Omega)} \leq c \|u\|_{L^p(\Omega)}.$$
\end{remark}

\begin{corollary}
 If $u,h\in L^p(\Omega)$ with $p> \max\{d/2, 1\}$, then 
 \begin{equation*}
  \frac{y_n - y}{t_n} \weak^* \eta \text{ in } L^\infty(\Omega),
 \end{equation*}
 which implies $\eta \in L^\infty(\Omega)$.
\end{corollary}

\begin{proof}
 By Lemma \ref{lem:inftylip} $(y_n - y)/t_n$ is bounded in $L^\infty(\Omega)$. Thus, there is a subsequence converging 
 weakly-$*$ to a element $\tilde\eta \in L^\infty(\Omega)$. 
 This subsequence therefore converges weakly in $L^2(\Omega)$ and in view of \eqref{eq:weakV} we find
 \begin{equation*}
  \int_\Omega \eta \, v \, dx =   \int_\Omega \tilde\eta \, v \, dx, \quad \forall \, v\in L^2(\Omega).
 \end{equation*}
 The fundamental lemma of the calculus of variations implies $\tilde\eta = \eta$ a.e.\ in $\Omega$. 
 Since the weak limit is therefore unique, a standard argument implies weak-$*$ convergence of the whole sequence as claimed.
\end{proof}

Based on the Lipschitz continuity of $S$ in Lemma \ref{lem:inftylip}, we can prove a first result towards an infinite dimensional 
counterpart to \eqref{eq:rabl2a}.

\begin{lemma}\label{lem:omega_rho}
 Assume that $u, h \in L^p(\Omega)$ with $p> \max\{d/2, 1\}$.
 Let moreover $\rho > 0$ be arbitrary and define --up to sets of measure zero--
 \begin{equation*}
  \AA_\rho := \{ x\in \Omega: y(x) \in [-\rho,\rho]\}.
 \end{equation*}
 Then for all $v \in V$ with $v(x) = 0$ a.e.\ in $\AA_\rho$ there holds
 \begin{equation*}
  \dual{\lambda}{v} = 0.
 \end{equation*}
\end{lemma}

\begin{proof}
 Let $\rho>0$ and $v\in V$ with $v(x) = 0$ a.e.\ in $\AA_\rho$ be arbitrary. 
 Thanks to Lemma \ref{lem:inftylip} we have
 \begin{equation}\label{eq:inftyconv}
  \|y_n - y\|_{L^\infty(\Omega)} \leq K\, t_n\, \|h\|_{L^p(\Omega)} \to 0. 
 \end{equation}
 Therefore, for almost all $x\in \Omega$ with $y(x) > \rho$, it follows that
 \begin{equation*}
  y_n(x) \geq y(x) - |y(x) - y_n(x)| \geq \rho - \|y - y_n\|_{L^\infty(\Omega)}  
  \geq \frac{\rho}{2} > 0, \quad \forall \, n \geq N_1,
 \end{equation*}
 with $N_1 \in \N$ depending on $\rho$ but not on $x$. Therefore, thanks to \eqref{eq:slacklike}, we have 
 for all $n \geq N_1$ that
 \begin{equation}\label{eq:qdiffnull}
  q_n(x) = \frac{y_n(x)}{|y_n(x)|} = 1 
  \quad \Rightarrow \quad 
  \frac{q_n(x) - q(x)}{t_n} = 0 \quad \text{f.a.a.\ } x \in \Omega \text{ with } y(x) > \rho,
 \end{equation}
 where we used that $q(x) = 1$ due to $y(x) > \rho > 0$.
 Completely analogously one can show the existence of $N_2 \in \N$, only depending on $\rho$, such that
 \begin{equation*}
  \frac{q_n(x) - q(x)}{t_n} = 0 \quad \text{f.a.a.\ } x \in \Omega \text{ with } y(x) < - \rho
 \end{equation*}
 for all $n \geq N_2$. Therefore, since $v(x) = 0$ a.e., where $y(x) \in [-\rho,\rho]$, we obtain
 \begin{equation*}
  \int_\Omega \frac{q_n - q}{t_n}\, v \, dx = 0 \quad \forall \, n \geq \max\{N_1, N_2\}. 
 \end{equation*}
 The convergence $(q_n - q)/t_n \weak \lambda$ in $V^*$ thus implies the assertion.
\end{proof}

The aim is now to drive $\rho$ in Lemma \ref{lem:omega_rho} to zero.
This however requires several additional assumptions. The first one covers the regularity of $y$ and $q$.

\begin{assumption}\label{assu:ycont}
 \begin{enumerate}
  \item\label{assu:ycont1} We assume that the solution $y = S(u)$ is continuous.
  \item\label{assu:ycont2} The slack function is continuous, i.e.\ $q\in C(\bar\Omega)$.
 \end{enumerate}
\end{assumption}

\begin{remark}
 Let us point out that Assumption \ref{assu:ycont}\eqref{assu:ycont1} is not restrictive at all. 
 Indeed, Lemma \ref{lem:slack} implies that $y$ solves $A y = u - q$ and, if $u\in L^2(\Omega)$, then $y$ thus solves a 
 second-order elliptic equation with right hand side in $L^2(\Omega)$. For problems of this type, standard 
 regularity theory yields continuity of the solution under mild assumptions on the data, see for instance \cite{Evans}.
 In contrast to this, Assumption \ref{assu:ycont}\eqref{assu:ycont2} cannot be guaranteed in general. 
 Nevertheless, multiple numerical observations indicate that $q$ is often continuous.
\end{remark}

If Assumption \ref{assu:ycont} is satisfied, i.e.\ if $y$ and $q$ have continuous representatives, then 
we can define the sets in Definition \ref{def:sets} in a pointwise manner, i.e., not only up to sets of zero measure. The sets arising in this way are denoted by the same symbols, and we 
always mean these sets in all what follows when writing $\AA$, $\II$, $\BB$ etc. 

\begin{lemma}\label{lem:inactivedist}
 Under Assumption \ref{assu:ycont} the sets $\II^+$ and $\II^-$ are strictly separated, i.e., there exists $\delta > 0$ such that
 \begin{equation*}
  \dist(\II^+, \II^-) := 
  \min\big\{|x-z|_{\R^d} : x\in \clos{\II^+}, z\in \clos{\II^-}\big\} > \delta. 
 \end{equation*}
\end{lemma}

\begin{proof}
 Since $\bar\Omega$ is compact, Assumption \ref{assu:ycont}\eqref{assu:ycont2} implies that $q$ is uniformly continuous. 
 From the slackness condition \eqref{eq:slacklike} we infer $q=1$ in $\II^+$ so that the uniform continuity of $q$ yields
 the existence of $\delta > 0$ with 
 \begin{equation}\label{eq:IplusB}
  q(x) \geq 1/2 \quad  \text{for all } x \in \II^+ + B(0,\delta).
 \end{equation}  
 Hence, due to $q = -1$ on $\II^-$ by \eqref{eq:slacklike}, this gives the assertion.
\end{proof}

In addition to Assumption \ref{assu:ycont}, we need the following rather restrictive assumption on the active set.

\begin{assumption}\label{assu:active}
 The active set $\AA = \{ x\in \Omega: y(x) = 0\}$ satisfies the following conditions:
 \begin{enumerate}
  \item\label{assu:active1} $\AA = \AA_1 \cup \AA_0$, where $\AA_1$ has positive measure and $\AA_0$ has zero capacity.
  \item\label{assu:active2} $\AA_1$ is closed and possesses non-empty interior. Moreover, it holds $\AA_1 = \clos{\interior(\AA_1)}$. 
  \item\label{assu:active3} For the set $\JJ:= \Omega\setminus \AA_1$ it holds
  \begin{equation}\label{eq:innererrand}
   \partial\JJ \setminus (\partial\JJ\cap\partial\Omega) = \partial\AA_1\setminus(\partial\AA_1\cap\partial\Omega),
  \end{equation}
  and both $\AA_1$ and $\JJ$ are supposed to have regular boundaries. 
  That is the connected components of $\JJ$ and $\AA_1$ have positive distance from each other and 
  the boundaries of each of them satisfies the cone condition.  
 \end{enumerate}
\end{assumption}

Figures \ref{fig:cap0} and \ref{fig:cap1} illustrate Assumption \ref{assu:active} in the two-dimensional case.
\begin{figure}[h!]
\centering
\begin{minipage}[t]{0.48\linewidth}
\centering
 \includegraphics[scale=0.5]{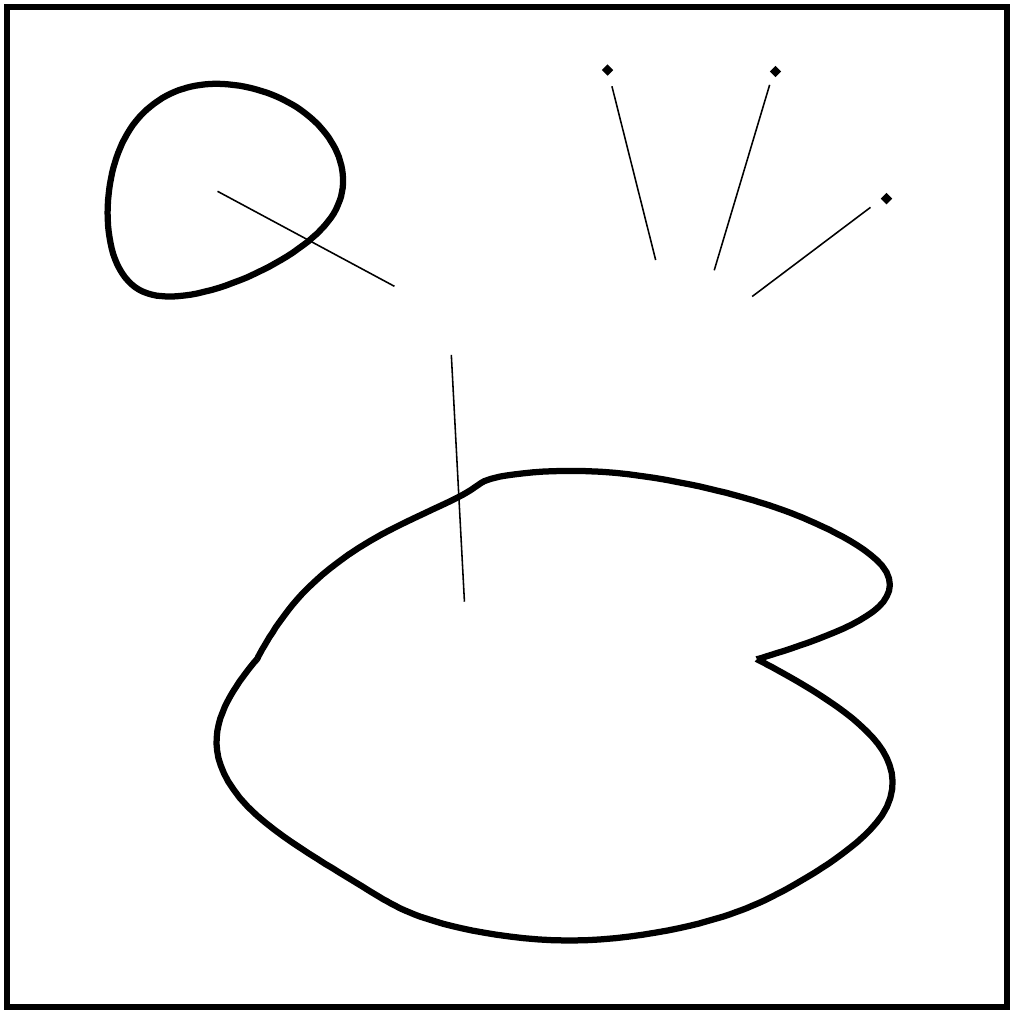}
 \put(-12,5){$\Omega$}
 \put(-90,100){$\AA_1$}
 \put(-53,97){$\AA_0$}
 \caption{Active set satistying Assumption \ref{assu:active}}\label{fig:cap0}
\end{minipage} 
\quad
\begin{minipage}[t]{0.48\linewidth}
\centering
 \includegraphics[scale=0.5]{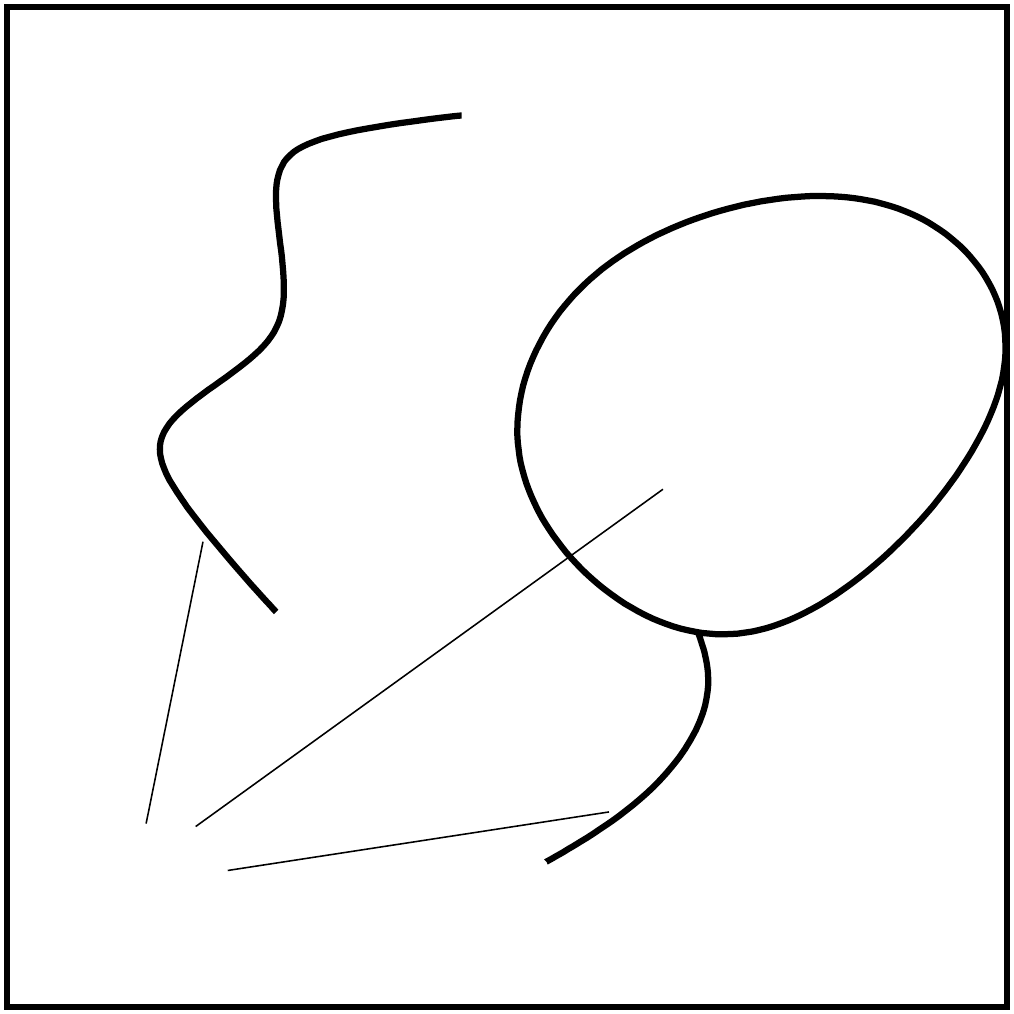}
 \put(-12,5){$\Omega$}
 \put(-130,17){$\AA$}
 \caption{Active set not feasible for Assumption \ref{assu:active}}\label{fig:cap1}
\end{minipage}
\end{figure}

With the help of Assumption \ref{assu:ycont} and \ref{assu:active} we can now prove the following 
infinite dimensional counterpart to \eqref{eq:rabl2a}:

\begin{lemma}\label{lem:lambdanull}
 Let $u, h \in L^p(\Omega)$, $p > \max\{d/2,1\}$, be given.
 Assume that $u$ is such that Assumptions \ref{assu:ycont} and \ref{assu:active} are fulfilled. Then
 \begin{equation*}
  \dual{\lambda}{v} = 0 \quad \text{for all } v \in V \text{ with } v(x) = 0 \text{ a.e.\ in }\AA
 \end{equation*}  
 holds true.
\end{lemma}

\begin{proof}
 Let $v \in V$ with $v(x) = 0$ a.e.\ in $\AA$ be arbitrary. 
 By Assumption \ref{assu:active}\eqref{assu:active3} there are linear and continuous trace operators $\tau_j: H^1(\Omega) \to L^2(\partial\JJ)$ 
 and $\tau_a: H^1(\Omega) \to L^2(\partial\AA_1)$. 
 Due to $v = 0$ a.e.\ in $\AA_1$, we have $\tau_a v = 0$ and, by \eqref{eq:innererrand} and $v\in V$, thus $\tau_j v = 0$.
 Since $\partial\JJ$ is regular, there exists a sequence
 $\{\varphi_k\}_{k\in \N} \subset C^\infty_0(\JJ)$ with $\varphi_k \to v$ in $H^1(\JJ)$, see e.g.\ \cite[Lemma 1.33]{GajewskiGroegerZacharias}.
 In particular it holds
 \begin{equation*}
  \omega_k := \supp(\varphi_k) \subset\subset \JJ.
 \end{equation*}
 We extend $\varphi_k$ by zero outside $\JJ$ to obtain a function in $C^\infty_0(\Omega)$, 
 which we denote by the same symbol for simplicity. Because of $v = 0$ a.e.\ in $\AA_1$ it follows that
 \begin{equation}\label{eq:vconv}
  \varphi_k \stackrel{k\to\infty}{\longrightarrow} v \quad \text{in }V. 
 \end{equation}
 By construction we have $\JJ \subset \II \cup \AA_0$.
 Since $\AA_0$ has zero capacity, there is a sequence $\{w_m\}_{m\in \N} \subset V$ and 
 a sequence of open neighborhoods of $\AA_0$, denoted by $\{\UU_m\}_{m\in \N}\subset \Omega$, such that 
 \begin{gather*}
  w_m\geq 0 \text{ a.e.\ in } \Omega, \quad w_m = 1 \text{ a.e.\ in }\UU_m, \quad
  w_m \stackrel{m\to\infty}{\longrightarrow} 0 \text{ in } H^1(\Omega).
 \end{gather*}
 Now let $k,m\in \N$ be fixed but arbitrary and define 
 \begin{equation*}
  \II_m^+ := (\omega_k \setminus \UU_m) \cap \II^+, \quad   \II_m^- := (\omega_k \setminus \UU_m) \cap \II^-.
 \end{equation*}
 Since $\UU_m$ is open, $\omega\setminus\UU_m$ is closed. Moreover, in view of $\JJ = \II \cup \AA_0$, 
 it holds $\omega_k\setminus\UU_m \subset \II$.
 Thus, Lemma \ref{lem:inactivedist} and the boundedness of $\Omega$ yield that $\II_m^+$ and $\II_m^-$ are compact.
 The continuity of $y$ therefore implies that there is $\xi \in \II_m^+$ such that
 \begin{equation*}
  y(\xi) = \min_{x\in \II_m^+} y(x)
 \end{equation*}
 and, due to $\xi \in \II^+$, one obtains $\rho_m^+ := y(\xi) > 0$. Analogously one derives 
 $\rho_m^- := \max_{x\in \II_m^-} y(x) < 0$. As in the proof of Lemma \ref{lem:omega_rho} one proves the existence of $N^+_m\in \N$ 
 such that for all $n\geq N^+_m$ there holds
 \begin{equation*}
  \frac{q_n(x) - q(x)}{t_n} = 0 \quad \text{f.a.a.\ $x\in \Omega$ with } y(x)\geq \rho_m^+, 
 \end{equation*}
 see \eqref{eq:qdiffnull}. Clearly, there is $N_m^- \in \N$ so that the same equation holds for every $n\geq N_m^-$ and almost all 
 $x\in \Omega$ with $y(x) \leq \rho_m^-$. Consequently, we obtain
 \begin{equation}\label{eq:diffqsupp}
  \frac{q_n - q}{t_n} = 0 \quad \text{a.e.\ in }\omega_k \setminus\UU_m = \II_m^+ \cup \II_m^-,
 \end{equation}
 provided that $n\geq N_m := \max\{N_m^+, N_m^-\}$.
 
 Thanks to \eqref{eq:diffqsupp} and $w_m = 1$ a.e.\ in $\UU_m$, it follows
 \begin{equation}\label{eq:umomega}
 \begin{aligned}
  \int_\Omega \frac{q_n - q}{t_n}\, \varphi_k \, w_m\, dx 
  &= \int_{\omega_k \setminus \UU_m} \frac{q_n - q}{t_n}\, \varphi_k\, w_m \,dx 
  + \int_{\UU_m} \frac{q_n - q}{t_n}\, \varphi_k \, w_m dx\\
  &= \int_{\UU_m} \frac{q_n - q}{t_n}\, \varphi_k \,dx 
  \qquad \forall\, n \geq N_m.
 \end{aligned}
 \end{equation}
 On the other hand $\varphi_k w_m \in V$ is a feasible test function for \eqref{eq:qdef} and \eqref{eq:qn}. 
 If we insert this test function and subtract the arising equation, then \eqref{eq:umomega} together with  
 H\"older's inequality and Sobolev embeddings yield
 \begin{equation*}
 \begin{aligned}
  &\int_{\UU_m} \frac{q_n - q}{t_n}\, \varphi_k \,dx 
  = \int_\Omega \frac{q_n - q}{t_n}\, \varphi_k \, w_m\, dx \\
  &\quad = - \int_\Omega \nabla \Big(\frac{y_n - y}{t_n}\Big) \cdot \nabla(\varphi_k w_m)\,dx + \int_\Omega h\,\varphi_k\,w_m\,dx\\
  &\quad \leq 2 \Big\|\frac{y_n - y}{t_n}\Big\|_{H^1(\Omega)} \|w_m\|_{H^1(\Omega)} \|\varphi_k\|_{W^{1,\infty}(\Omega)} 
  + c\,\|h\|_{L^2(\Omega)} \|w_m\|_{H^1(\Omega)} \|\varphi_k\|_{H^1(\Omega)}
 \end{aligned}
 \end{equation*}
 for all $n\geq N_m$. Therefore, in view of \eqref{eq:diffqsupp}, the weak convergence (and thus boundedness) of $(y_n - y)/t_n$ gives
 \begin{equation*}
  \int_\Omega \frac{q_n - q}{t_n}\, \varphi_k \,dx 
  = \int_{\UU_m} \frac{q_n - q}{t_n}\, \varphi_k \,dx 
  \leq c\,\|w_m\|_{H^1(\Omega)} \|\varphi_k\|_{W^{1,\infty}(\Omega)} 
 \end{equation*}
 for all $n\geq N_m$ and thus 
 \begin{equation*}
  \dual{\lambda}{\varphi_k} = \lim_{n\to \infty} \int_\Omega \frac{q_n - q}{t_n}\, \varphi_k \,dx 
  \leq c\,\|w_m\|_{H^1(\Omega)} \|\varphi_k\|_{W^{1,\infty}(\Omega)}.
 \end{equation*}  
 Due to $w_m \to 0$ in $H^1(\Omega)$, passing to the limit $m\to\infty$ yields $\dual{\lambda}{\varphi_k} \leq 0$. 
 The above arguments also apply to $-\varphi_k$ so that $\dual{\lambda}{\varphi_k} = 0$. 
 Since $k\in \N$ was arbitary, this equation holds for every $k\in \N$ and thus we can pass to the limit 
 $k\to \infty$. The convergence in \eqref{eq:vconv} then gives the assertion.
\end{proof}

Similarly to \eqref{eq:conefinite}, we define 
\begin{equation}\label{eq:KKae}
\begin{aligned}
 \KK(y) &:= \{v\in V: \;\, v(x) = 0 \text{ a.e.\ in } \AA_s,\; v(x)q(x) \geq 0 \text{ a.e.\ in } \BB\}\\
 &=
 \begin{aligned}[t]
  \{v\in V: \;\, & v(x) = 0 \text{ a.e., where } |q(x)| < 1,\\ 
  & v(x)q(x) \geq 0 \text{ a.e., where } |q(x)| = 1 \text{ and } y(x) = 0\}
 \end{aligned} 
\end{aligned}
\end{equation}
This set will be the feasible set of the VI belonging to the directional derivative of $S$ (see Theorem \ref{thm:ablvi} below).
As seen in the proof of Theorem \ref{thm:rablfinite}, in the finite dimensional setting, there holds $\lambda^\top v \leq 0$ for all $v\in K(y)$, 
see \eqref{eq:lambdasign}. 
The infinite dimensional analogon is also true, provided that Assumptions \ref{assu:ycont} and \ref{assu:active} hold, 
as the following lemma shows.

\begin{lemma}\label{lem:signlambda}
 Let $u,h \in L^p(\Omega)$ with $p > \max\{d/2,1\}$ be given, and assume that $u$ is such that 
 Assumptions \ref{assu:ycont} and \ref{assu:active} are fulfilled. Then there holds 
 \begin{equation*}
  \dual{\lambda}{v} \leq 0 \quad \text{for all } v\in \KK(y).
 \end{equation*}
\end{lemma}

\begin{proof}
 Let $v\in \KK(y)$ be fixed but arbitrary.
 Due to $\AA_s \cup \BB \cup \II = \Omega$ and $v(x) = 0$ a.e.\ in $\AA_s$, we obtain
 \begin{equation}\label{eq:intest}
  \int_\Omega \frac{q_n - q}{t_n}\,v\,dx
  = \int_{\BB^+} \frac{q_n - 1}{t_n}\,v\,dx 
  + \int_{\BB^-} \frac{q_n + 1}{t_n}\,v\,dx 
  + \int_{\II} \frac{q_n - q}{t_n}\,v\,dx,
 \end{equation}
 Since $q_n \in [-1,1]$ a.e.\ in $\Omega$ and $q \,v \geq 0$ a.e.\ in $\BB$, which implies $v \geq 0$ a.e.\ in $\BB^+$ and  
 $v \leq 0$ a.e.\ in $\BB^-$, we can further estimate
 \begin{equation*}
  \int_\Omega \frac{q_n - q}{t_n}\,v\,dx \leq \int_{\II} \frac{q_n - q}{t_n}\,v\,dx = \int_{\JJ} \frac{q_n - q}{t_n}\,v\,dx, 
 \end{equation*}
 where $\JJ$ is the set from Assumption \ref{assu:active}\eqref{assu:active3}. For the last equality we used that 
 $\JJ = \II \cup \AA_0$ and $\AA_0$ has zero capacity, thus zero Lebesgue-measure. 

 We now prove that $\JJ = \JJ^+ \cup \JJ^-$, where $\JJ^+$ and $\JJ^-$ possess regular boundaries and coincide with 
 $\II^+$ and $\II^-$ up to sets of zero capacity.
 For this purpose, we show $\clos{\JJ} = \clos{\II}$. Due to $\II \subset \JJ$, we clearly have $\clos{\II}\subseteq\clos{\JJ}$. 
 Let $\xi \in \clos{\JJ}$ be arbitrary. Then there is a sequence $\{x_k\}_{k\in \N}\subset \JJ$ so that $x_k \to \xi$. 
 If $\{x_k\}$ contains a subsequence in $\II$, we immediately obtain $\xi \in \clos{\II}$. So assume the contrary, 
 i.e., in view of $\JJ = \II \cup \AA_0$, $x_k\in \AA_0$ for all $k\in \N$ sufficiently large. W.l.o.g.\ we assume $\{x_k\}\subset \AA_0$ 
 for the whole sequence. Since $\AA_0$ has zero capacity, thus zero measure, there is, for each $x_k$, 
 a sequence $\{x_k^{(m)}\}_{m\in \N} \subset \Omega \setminus \AA_0$ with $x^{(m)}_k \to x_k$ for $m\to \infty$. 
 Since $x_k^{(m)} \notin \AA_0$, we have either $x^{(m)}_k \in \AA_1$ or  $x^{(m)}_k \in \II$. If $\{x_k^{(m)}\}$ would contain a 
 subsequence in $\AA_1$, then the closedness of $\AA_1$ would imply $x_k \in \AA_1$ in contradiction to $x_k \in \AA_0$. 
 Thus we may w.l.o.g.\ assume that $\{x_k^{(m)}\}\subset \II$. Therefore, there is a diagonal sequence $\{x_k^{(m(k))}\} \subset \II$ 
 converging to $\xi$, which gives $\xi \in \clos{\II}$. Hence we have shown
 \begin{equation*}
  \clos{\JJ}=\clos{\II} = \clos{\II^+} \cup \clos{\II^-}
 \end{equation*}  
 with $\II^+$ and $\II^-$ as defined in \eqref{eq:defsets}. Since $\clos{\II^+}$ and $\clos{\II^-}$ have positive distance from 
 each other by Lemma \ref{lem:inactivedist}, there exist sets $\JJ^+, \JJ^-$ such that $\JJ^+ \cup \JJ^- = \JJ$ and 
 $\dist(\JJ^+, \JJ^-) > \delta$. Moreover, thanks to Lemma \ref{lem:inactivedist} and 
 $\JJ = \II \cup \AA_0$ with $\capa(\AA_0) = 0$, $\JJ^+$ differs from $\II^+$ only by a set of zero capacity and the same holds for 
 $\JJ^-$ and $\II^-$. Finally, because of $\dist(\JJ^+, \JJ^-) > \delta$, Assumption \ref{assu:active}\eqref{assu:active3} yields that 
 $\JJ^+$, $\JJ^-$, $\Omega \setminus \JJ^+$, and $\Omega\setminus\JJ^-$ possess regular boundaries. 
 (This actually implies that $\JJ^\pm = \interior(\clos{\II^\pm}).$)
 
 Since $\JJ^+$ differs from $\II^+$ only on a set of zero measure, the definition of $\II^+$ and the slackness condition \eqref{eq:slacklike} 
 imply $q = 1$ a.e.\ in $\JJ^+$, and analogously $q = -1$ a.e.\ in $\JJ^-$. Thus \eqref{eq:intest} can be further estimated by
 \begin{align}
  \int_\Omega \frac{q_n - q}{t_n}\,v\,dx 
  &\leq \int_{\JJ^+} \underbrace{\frac{q_n - 1}{t_n}}_{\leq 0}\,\underbrace{\max\{0,v\}}_{\geq 0}\,dx 
  + \int_{\JJ^+} \frac{q_n - q}{t_n}\,\min\{0,v\}\,dx \nonumber\\
  &\quad + \int_{\JJ^-} \underbrace{\frac{q_n + 1}{t_n}}_{\geq 0}\,\underbrace{\min\{0,v\}}_{\geq 0}\,dx 
  + \int_{\JJ^-} \frac{q_n - q}{t_n}\,\max\{0,v\}\,dx \nonumber\\
  & \leq \int_{\JJ^+} \frac{q_n - q}{t_n}\,\min\{0,v\}\,dx + \int_{\JJ^-} \frac{q_n - q}{t_n}\,\max\{0,v\}\,dx. \label{eq:intest2}
 \end{align}
 Next we show that $\min\{0,v\} \in H^1_0(\JJ^+)$ and $\max\{0,v\} \in H^1_0(\JJ^-)$. The proof of Lemma \ref{lem:inactivedist} shows 
 \begin{equation}\label{eq:incl}
  \big(\II^+ + B(0,\varepsilon)\big) \setminus \II^+ 
  \subset \{x\in \Omega: q(x) \geq 1/2,\; y(x) = 0\} \subset \AA_s \cup \BB^+,
 \end{equation}
 see \eqref{eq:IplusB}.
 Because of $v\in \KK(y)$ we have $q\,v \geq 0$ a.e.\ in $\AA_s \cup \BB^+$ and thus \eqref{eq:incl} gives 
 $v \geq 0$ a.e.\ in $(\II^+ + B(0,\varepsilon)\big) \setminus \II^+$. Since $\II^+$ and $\JJ^+$ differ only up to a set of zero 
 measure, we thus get
 \begin{equation*}
  \min\{0,v\} = 0 \quad \text{a.e.\ in } (\JJ^+ + B(0,\varepsilon)) \setminus \JJ^+.
 \end{equation*}
 The regularity of $\partial\JJ^+$ and $\partial(\Omega\setminus\JJ^+)$ therefore gives
 \begin{equation*}
  \min\{0,v(x)\} = 0\quad \text{a.e.\ on }\partial\JJ^+,
 \end{equation*}
 and thus $\min\{0,v\} \in H^1_0(\JJ^+)$. An analogous argument shows that $\max\{0,v\} \in H^1_0(\JJ^-)$.
 Due to the zero trace and the regularity of $\partial\JJ^+$ by Assumption \ref{assu:active}\eqref{assu:active3}, 
 we can extend $\min\{0,v\}$ by zero outside $\JJ^+$ to obtain a function in $V$, 
 i.e., $\chi_{\JJ^+}\min\{0,v\} \in V$, where $\chi_{\JJ^+}$ denotes the characteristic function of $\JJ^+$.
 Thus the weak convergence $(q_n - q)/t_n \weak \lambda$ in $V^*$ gives
 \begin{equation*}
  \int_{\JJ^+} \frac{q_n - q}{t_n}\,\min\{0,v\}\,dx 
  = \int_{\Omega} \frac{q_n - q}{t_n}\,\chi_{\JJ^+}\min\{0,v\}\,dx 
  \to \dual{\lambda}{\chi_{\JJ^+}\min\{0,v\}}.
 \end{equation*}  
 Since $\chi_{\JJ^+}\min\{0,v\} = 0$ a.e.\ in $\AA \subset \Omega\setminus\JJ^+$, Lemma \ref{lem:lambdanull} yields 
 $\dual{\lambda}{\chi_{\JJ^+}\min\{0,v\}} = 0$. Analogously
 \begin{equation*}
  \int_{\JJ^-} \frac{q_n - q}{t_n}\,\max\{0,v\}\,dx 
  \to \dual{\lambda}{\chi_{\JJ^-}\max\{0,v\}} = 0
 \end{equation*}  
 is obtained. Therefore, in view of \eqref{eq:intest2}, we finally arrive at $\dual{\lambda}{v} \leq 0$
 and, since $v\in \KK(y)$ was arbitrary, this proves the assertion.
\end{proof}

Now we are finally in the position to prove the main result of this section covering the ''weak directional differentiability`` of 
the solution operator associated with the VI in \eqref{eq:vi}.

\begin{theorem}\label{thm:ablvi}
 Let $u,h \in L^p(\Omega)$ with $p > \max\{d/2,1\}$ be given. Suppose further that Assumptions \ref{assu:ycont} 
 and \ref{assu:active} are fulfilled by $y = S(u)$ and the associated slack variable $q$. Then there holds 
 \begin{equation}\label{eq:weaklim}
  \frac{S(u + t\,h) - S(u)}{t} \weak \eta \quad \text{in } V, \quad \text{as } t \searrow 0,
 \end{equation}
 where $\eta \in V$ solves the following VI of first kind:
 \begin{equation}\label{eq:ablvi}
 \begin{aligned}
  \eta \in \KK(y),\quad \dual{A\eta}{v-\eta} \geq \dual{h}{v-\eta} \quad \forall\, v\in \KK(y)
 \end{aligned}
 \end{equation}
 with $\KK(y)$ as defined in \eqref{eq:KKae}.
\end{theorem}

\begin{proof}
 Lemma \ref{lem:etafeas} yields $\eta \in \KK(y)$. Furthermore, since $A \eta + \lambda = h$, 
 Lemmas \ref{lem:etalambdanull} and \ref{lem:signlambda} give
 \begin{equation*}
  \dual{A\eta}{v-\eta} - \dual{h}{v-\eta} 
  = \dual{\lambda}{\eta} - \dual{\lambda}{v} \geq 0
 \end{equation*}
 for all $v\in \KK(y)$, which is just the VI in \eqref{eq:ablvi}.

 Since $\KK(y)$ is nonempty, convex, and closed and $A$ is bounded and coercive, 
 standard arguments yields existence and uniqueness for this VI of first kind. 
 Thus the weak limit $\eta$ is unique, which implies the weak convergence of the whole sequence.
\end{proof}

\begin{definition}
 With a little abuse of notation we call the weak limit $\eta$ in \eqref{eq:weaklim} \emph{weak directional derivative} 
 and denote it by $\eta = S_w'(u;h)$.
\end{definition}

\begin{remark}
 If $\BB$ has zero measure, then $\KK(y)$ turns into
 \begin{equation*}
  \KK(y) =\{v\in V: \;\, v(x) = 0 \text{ a.e.\ in } \AA_s\},
 \end{equation*}
 i.e., a linear and closed subspace of $V$. Thus, in this case, \eqref{eq:ablvi} becomes an equation. 
 If $\AA_s$ possesses a regular boundary, then this equation is equivalent to 
 \begin{equation*}
  A \eta  = h \quad \text{in } \II
  \quad\text{and}\quad \eta = 0\quad \text{a.e.\ in }\AA = \AA_s.
 \end{equation*}
\end{remark}

\begin{remark}
 It is very likely that Theorem \ref{thm:ablvi} could be proven without the restrictive Assumption \ref{assu:active}, 
 if the weak limit $\eta$ would satisfy the conditions in \eqref{eq:etaAs} and \eqref{eq:etaB} not only almost everywhere, 
 but \emph{quasi-everywhere} in $\Omega$. In this case, the feasible set of \eqref{eq:ablvi} would read
 \begin{equation*}
 \begin{aligned}
  \KK := \{v\in V: \;\, & v(x) = 0 \text{ q.e., where } |q(x)| < 1,\\ 
  & v(x)q(x) \geq 0 \text{ q.e., where } |q(x)| = 1 \text{ and } y(x) = 0\}.
 \end{aligned}
 \end{equation*}
 However, unfortunately, so far we have neither been able to show that \eqref{eq:etaqae} holds quasi everywhere, nor to establish a 
 counterexample which demonstrates that this is wrong in general. This question gives rise to future research.
\end{remark}

\section{Bouligand stationarity}\label{sec:bouli}

With the differentiability result of Theorem \ref{thm:ablvi} at hand, it is now straightforward to establish 
first-order optimality conditions in purely primal form for optimization problems governed by \eqref{eq:vi}. 
To be more precise, we consider an optimization problem of the form
\begin{equation}\label{eq:optcontrol}
 \left.
 \begin{aligned}
  \min \quad & J(y,u)\\
  \text{s.t.} \quad &
  \dual{A y}{v-y} + \int_\Omega |v|\,dx - \int_\Omega |y|\,dx \geq \dual{u}{v-y}
  \quad \forall \, v\in V\\
  \text{and}\quad & u \in \Uad,
 \end{aligned}
 \qquad \right\}
\end{equation}
where $\Uad\subset L^p(\Omega)$, $p > \max\{d/2,1\}$,  is nonempty, closed, and convex.

As shown in \cite[Lemma 3.9]{HerzogMeyerWachsmuth}, weak convergence of the difference quotient associated with the 
control-to-state mapping $S: u \mapsto y$ is sufficient to prove that the reduced objective, defined by
\begin{equation*}
 j: L^p(\Omega) \to \R, \quad j(u) := J(S(u),u),
\end{equation*}
is directionally differentiable. This allows us to formulate the following purely primal optimality conditions, 
which, in case of optimal control of VIs of first kind, are known as Bouligand stationarity conditions.

\begin{theorem}\label{thm:bstat}
 Let $p > \max\{d/2,1\}$ and assume that $J$ is Fr\'echet-differentiable from $V\times L^p(\Omega)$ to $\R$. 
 Suppose moreover that $\bar u \in \Uad$ is a local optimal solution of \eqref{eq:optcontrol}, such that 
 $\bar y = S(\bar u)$ and the associated slack variable $\bar q$ satisfy Assumptions \ref{assu:ycont} and \ref{assu:active}.  
 Then the following primal stationarity conditions are fulfilled:
 \begin{equation}\label{eq:noc1}
  \partial_y J(\bar y, \bar u) \eta + \partial_u J(\bar y, \bar u)(u - \bar u) \geq 0
  \quad \forall \, u \in \Uad,
 \end{equation}
 where $\eta\in V$ solves \eqref{eq:ablvi} with $\KK(y) = \KK(\bar y)$ and $h = u - \bar u$.
\end{theorem}

\begin{proof}
 As mentioned above, \cite[Lemma 3.9]{HerzogMeyerWachsmuth} and Theorem \ref{thm:ablvi} imply that $u \mapsto j(u)$ 
 is directionally differentiable in every direction $h \in L^p(\Omega)$ with directional derivative 
 $j'(\bar u;h) = \partial_y J(\bar y, \bar u)S_w'(\bar u;h) + \partial_u J(\bar y, \bar u)h$. 
 Local optimality of $\bar u$ yields $j'(\bar u; u - \bar u)\geq 0$, which is the assertion.
\end{proof}

Next we derive a variant of the above optimality condition based on the cone tangent to the admissble set of \eqref{eq:optcontrol}. 
As a result, we obtain an optimality condition which can be interpreted as the counterpart of the implicit programming 
approach in the discussion of finite dimensional MPECs, see \cite[Section 3.3]{LuoPangRalph}. 
Note that such similarities have already been observed in \cite{HerzogMeyerWachsmuth}.

\begin{lemma}\label{lem:convrabl}
 Assume that $\bar u\in L^p(\Omega)$, $p > \max\{d/2,1\}$, is such that Assumptions \ref{assu:ycont} and \ref{assu:active} 
 are fulfilled. Suppose moreover that the sequences $\{u_n\}\subset L^p(\Omega)$ and $\{t_n\}\subset \R^+$ satisfy
 \begin{equation*}
  t_n \searrow 0, \quad \frac{u_n - \bar u}{t_n} \weak h \quad \text{in } L^p(\Omega).
 \end{equation*}
 Then 
 \begin{equation*}
  \frac{S(u_n) - S(\bar u)}{t_n} \weak S'_w(\bar u;h) \quad \text{in } V.
 \end{equation*}
\end{lemma}

\begin{proof}
 By adding a zero we obtain
 \begin{equation*}
 \begin{aligned}
  \frac{S(u_n) - S(\bar u)}{t_n} 
  = \frac{S(u_n) - S(\bar u + t_n\, h)}{t_n} + \frac{S(\bar u + t_n\, h) - S(\bar u)}{t_n}.
 \end{aligned}
 \end{equation*}
 While the latter addend converges weakly to $S'_w(\bar u;h)$ by Theorem \ref{thm:ablvi}, 
 the Lipschitz continuity of $S$ by Lemma \ref{lem:lipschitz} yields for the first addend that
 \begin{equation*}
  \Big\|\frac{S(u_n) - S(\bar u + t_n\, h)}{t_n}\Big\|_V
  \leq L\,\Big\|\frac{u_n - \bar u}{t_n} - h\Big\|_{V^*} \to 0,
 \end{equation*}
 where we used the compactness of the embedding $L^p(\Omega) \embed V^*$.
\end{proof}

We define the tangent cone to the admissible set of \eqref{eq:optcontrol} as follows:

\begin{definition}[Tangent cone]
 For given $u \in \Uad$ we define the tangent cone at $u$ by
 \begin{equation*}
  \TT(u) := 
  \begin{aligned}[t]
   &\Big\{(\eta,h) \subset V \times L^p(\Omega):
   \exists \; \{u_n\}_{n\in\N} \subset \Uad, \, \{t_n\} \subset \R^+ \text{ such that }\\
   &\qquad\qquad \frac{u_n - u}{t_n} \weak h \text{ in } L^p(\Omega) \quad\text{and}\quad
   \frac{S(u_n) - S(u)}{t_n} \weak \eta \text{ in } V\Big\}. 
  \end{aligned}   
 \end{equation*}
\end{definition}

Since the VI in \eqref{eq:optcontrol} is uniquely solvable such that $y$ is determined by $u$, this cone coincides with 
the standard tangent cone in finite dimensions, except that we replace strong by weak convergence.
Next consider the VI in \eqref{eq:ablvi} associated with the directional derivative of $S$ at $\bar u$. 
Due to the coercivity of $A$, this VI does clearly not only admit a unique solution for right hand sides in $L^p(\Omega)$, 
but also for inhomogeneities in $V^*$. We denote the associated solution operator by $G: V^* \to V$, i.e.
\begin{equation}\label{eq:Gdef}
 \eta = G(h) \quad :\Longleftrightarrow\quad  \eta \in \KK(\bar y), \quad
 \dual{A \eta}{v - \eta} \geq \dual{h}{v - \eta} \quad \forall\,v\in \KK(\bar y).
\end{equation}
Furthermore, owing again to the coercivity of $A$ this operator is Lipschitz continuous, i.e.
\begin{equation}\label{eq:Glip}
 \|G(h_1) - G(h_2)\|_V \leq \frac{1}{\alpha}\, \|h_1 - h_2\|_{V^*} \quad \forall\, h_1, h_2\in V^*,
\end{equation}
where $\alpha$ is the coercivity constant of $A$. This enables us to show the following

\begin{theorem}
 Suppose that the assumptions of Theorem \ref{thm:bstat} are fulfilled with a local optimum $\bar u\in \Uad$ of \eqref{eq:optcontrol}. 
 Then there holds
 \begin{equation}\label{eq:noc2}
  \partial_y J(\bar y, \bar u) \eta + \partial_u J(\bar y, \bar u)h \geq 0
  \quad \forall \, (\eta,h) \in \TT(\bar u).
 \end{equation}
\end{theorem}

\begin{proof}
 If $h_n \weak h$ in $L^p(\Omega)$ and consequently $h_n \to h$ in $V^*$, then \eqref{eq:Glip} gives $G(h_n) \to G(h)$ in $V$.
 Since $G(h) = S_w'(\bar u; h)$ for $h\in L^p(\Omega)$, this implies that $L^p(\Omega) \ni h \mapsto S'_w(\bar u;h) \in V$ 
 is completely continuous.
 Now let $(\eta,h) \in \TT(\bar u)$ be arbitrary. Hence there is $\{u_n\}\in\Uad$ so that $(u_n - \bar u)/t_n \weak h$ in $L^p(\Omega)$. 
 As seen above, $S_w'(\bar u; .)$ is the solution operator of a VI of first kind with the cone $\KK(\bar y)$ as feasible set. 
 Hence, $S_w'(\bar u; .)$ is positively homogeneous such that Theorem \ref{thm:bstat} yields
 \begin{equation}\label{eq:gradineq}
  \partial_y J(\bar y, \bar u) S_w'\Big(\bar u;\frac{u_n - \bar u}{t_n}\Big) 
  + \partial_u J(\bar y, \bar u)\Big(\frac{u_n - \bar u}{t_n}\Big) \geq 0.
 \end{equation}   
 The complete continuity of $S_w'(\bar u; .)$ together with Lemma \ref{lem:convrabl} implies
 \begin{equation*}
  S_w'\Big(\bar u;\frac{u_n - \bar u}{t_n}\Big) \to S_w'(\bar u; h) = \eta \quad \text{in } V.
 \end{equation*}
 Due to the weak continuity of $\partial_u J(\bar y, \bar u)$ the second addend in \eqref{eq:gradineq} converges to 
 $\partial_u J(\bar y, \bar u)h$, which completes the proof.
\end{proof}

\section{Strong stationarity}

In this section we aim at deriving optimality conditions which, in contrast to the ones presented in Section \ref{sec:bouli},
also involve dual variables. Given the differentiability result and the Bouligand stationarity conditions in Theorem \ref{thm:bstat}, 
we can follow the lines of \cite{MignotPuel1984}. For this purpose we have to require the following assumptions concerning the 
quantities in the optimal control problem \ref{eq:optcontrol}:

\begin{assumption}\label{assu:strong1}
 We suppose that $U_{\textup{ad}} = L^2(\Omega)$. Moreover, $J$ is continously Fr\'echet-differentiable from $V\times L^2(\Omega)$ 
 to $\R$.
\end{assumption}

In order to be able to utilize our differentiability result we furthermore assume the following:

\begin{assumption}\label{assu:strong2}
 Assume that $\bar u$ is a local optimum such that the associated state $\bar y$ and the associated slack variable $\bar q$ 
 satisfy Assumptions \ref{assu:ycont} and \ref{assu:active}. 
\end{assumption}

\begin{lemma}\label{lem:bstatV}
 Under Assumptions \ref{assu:strong1} and \ref{assu:strong2} there exists a $\bar p\in V$ such that
 \begin{equation*}
  \partial_y J(\bar y, \bar u)G(h) - \dual{\bar p}{h} \geq 0 \quad \forall\, h\in V^*
 \end{equation*}
 with $G$ as defined in \eqref{eq:Gdef}.
\end{lemma}

\begin{proof}
 By Theorem \ref{thm:bstat} and $S'_w(\bar u;h) = G(h)$ for $h\in L^2(\Omega)$, there holds
 \begin{equation}\label{eq:gradvi}
  \partial_y J(\bar y,\bar u)G(h) + \partial_u J(\bar y,\bar u)h \geq 0 \quad \forall\, h\in L^2(\Omega).
 \end{equation}
 which, together with \eqref{eq:Glip}, gives in turn
 \begin{equation*}
  \partial_u J(\bar y,\bar u)h \leq \|\partial_y J(\bar y,\bar u)\|_{V^*} \,\frac{1}{\alpha}\, \|h\|_{V^*} 
  \quad \forall\, h\in L^2(\Omega).
 \end{equation*}  
 Therefore, by the Hahn-Banach theorem, the linear functional $\partial_u J(\bar y, \bar u): L^2(\Omega) \to \R$ 
 can be extended to a linear and bounded functional on $V^*$, which we identify with a function $\bar p\in V$, i.e.
 \begin{equation*}
  \dual{\bar p}{h} = - \partial_u J(\bar y, \bar u) h \quad \forall\, h \in L^2(\Omega).
 \end{equation*}
 The density of $L^2(\Omega) \embed V^*$ in combination with \eqref{eq:gradvi} then gives the assertion.
\end{proof}

Next define $q\in V$ as solution of 
\begin{equation*}
 \dual{A^* q}{v} = \dual{\partial_y J(\bar y,\bar u)}{v} \quad \forall\, v\in V,
\end{equation*}
which is well defined because of the coercivity of $A$. Furthermore, we introduce the operator $\Pi: V \to \KK(\bar y)$ by
\begin{equation*}
 \Pi := G \circ A.
\end{equation*}
Note that $\Pi$ can be interpreted as $A$-projection on $\KK(\bar y)$. It is straightforward to see the following properties of $\Pi$:
\begin{equation}\label{eq:idem}
\begin{aligned}
 & \text{$\Pi$ as well as $I - \Pi$ are idempotent,}\\
 & \Pi \circ (I-\Pi) = (I-\Pi) \circ \Pi = 0,
\end{aligned}
\end{equation}
and, as $\KK(\bar y)$ is a convex cone, 
\begin{equation}\label{eq:QPsenkrecht}
 \dual{A(I-\Pi)\xi}{\Pi(\xi)} = 0 \quad \forall\, \xi \in V.
\end{equation}
Moreover by construction, we find $G = \Pi \circ A^{-1}$. Thus Lemma \ref{lem:bstatV} implies for every $h \in V^*$ that
\begin{equation}\label{eq:gradineqmod}
\begin{aligned}
 0 &\leq \partial_y J(\bar y,\bar u)G(h) - \dual{\bar p}{h}\\
 &= \dual{G(h)}{A^*q} - \dual{A A^{-1} h}{\bar p}\\
 &= \dual{A\Pi(A^{-1}h)}{q - \bar p} - \dual{A(I - \Pi)(A^{-1}h)}{\bar p}\\
 &= 
 \begin{aligned}[t]
  & \dual{\Pi(A^{-1} h)}{A^*(q - \bar p)}\\ 
  & - \dual{A(I - \Pi)(A^{-1} h)}{\Pi(\bar p)} - \dual{A(I - \Pi)(A^{-1} h)}{(I-\Pi)(\bar p)}.
 \end{aligned}
\end{aligned}
\end{equation}
If we insert $h = A(I-\Pi)\bar p \in V^*$, then \eqref{eq:idem} and \eqref{eq:QPsenkrecht} yield
\begin{equation*}
 \dual{A(I-\Pi)\bar p}{(I-\Pi)\bar p} \leq 0.
\end{equation*}
The coercivity of $A$ then implies $\bar p = \Pi(\bar p)$ and thus $\bar p\in \KK(\bar y)$, i.e.
\begin{equation*}
\begin{aligned}
 \bar p(x) &= 0 & & \text{a.e., where } |\bar q(x)| < 1,\\ 
 \bar p(x)\bar q(x) &\geq 0 & & \text{a.e., where } |\bar q(x)| = 1 \text{ and } \bar y(x) = 0.
\end{aligned} 
\end{equation*}
Next we define $z\in V$ by 
\begin{equation}\label{eq:zgl}
 \dual{A z}{v} = \dual{v}{A^*(\bar p - q)}\quad \forall\, v\in V
\end{equation}
and insert $h = A\Pi(z)\in V$ in \eqref{eq:gradineqmod}. Together with \eqref{eq:idem}, \eqref{eq:zgl}, and \eqref{eq:QPsenkrecht}, 
we obtain in this way
\begin{equation*}
 0\leq \dual{\Pi(z)}{A^*(q - \bar p)} = -\dual{Az}{\Pi(z)} = -\dual{A\Pi(z)}{\Pi(z)}
\end{equation*}
so that $\Pi(z) = G(Az)= 0$ by the coercivity of $A$. Consequently the definition of $G$ in \eqref{eq:Gdef} leads to 
\begin{equation*}
 \dual{A z}{v} \leq 0 \quad \forall\, v\in \KK(\bar y)
 \quad \Longrightarrow \quad 
 \dual{A^*\bar p}{v} \leq \dual{A^* q}{v} = \dual{\partial_y J(\bar y,\bar u)}{v} \quad \forall\, v\in \KK(\bar y).
\end{equation*}
By defining $\bar\mu := g'(\bar y) - A^*\bar p \in V^*$ we therefore arrive at
\begin{equation*}
\begin{aligned}
 A^* \bar p &= \partial_y J(\bar y,\bar u) - \mu \quad \text{in } V^*\\
 \dual{\bar\mu}{v} &\geq 0 \quad \forall\, v\in \KK(\bar y).
\end{aligned}
\end{equation*}
All in all we have thus proven the following:

\begin{theorem}
 Assume that Assumption \ref{assu:strong1} holds. Suppose moreover that $\bar u$ is a local optimum which satisfies 
 Assumption \ref{assu:strong2}. Then there exists an adjoint state $\bar p\in V$ and a multiplier $\mu \in V^*$ such that 
 the following \emph{strong stationarity system} is fulfilled:
 \begin{subequations}\label{eq:strongstat}
 \begin{gather}
  A \bar y + \bar q = \bar u \quad \text{in } V^* \label{eq:state1}\\
  \bar q(x) \,\bar y(x) = |\bar y(x)|, \quad |\bar q(x)| \leq 1\quad \text{a.e.\ in }\Omega \label{eq:state2}\\[1mm]
  A^* \bar p = \partial_y J(\bar y,\bar u) - \mu \quad \text{in } V^* \label{eq:adjoint}\\
  \bar p\in \KK(\bar y),\quad \dual{\bar\mu}{v} \geq 0 \quad \forall\, v\in \KK(\bar y) \label{eq:compl}\\[1mm]
  \bar p + \partial_u J(\bar y,\bar u) = 0\label{eq:gradeq}
 \end{gather} 
 \end{subequations}
 with $\KK(\bar y)$ as defined in \eqref{eq:KKae}.
\end{theorem}

\begin{remark}
 A comparable result for optimal control problems governed by VIs of the first kind is known as strong stationarity conditions, 
 see \cite{hintermuller2009mathematical}. This is why we have chosen the same terminology here.  
\end{remark}

\begin{remark}
 We compare the optimality system \eqref{eq:strongstat} with results from \cite{Delosreyes2009} obtained via regularization and 
 subsequent limit analysis. The optimality system obtained in \cite{Delosreyes2009} coincides with \eqref{eq:strongstat} 
 except that \eqref{eq:compl} is replaced by 
 \begin{equation}\label{eq:cstat}
  \dual{\bar\mu}{\bar p} \geq 0, \quad \dual{\bar\mu}{\bar y} = 0.
 \end{equation}
 However, thanks to the definition of $\KK(\bar y)$ and $\pm\bar y \in \KK(\bar y)$, these relations are an immediate 
 consequence of \eqref{eq:compl}. The optimality system in \eqref{eq:strongstat} is therefore sharper compared 
 to the one obtained via regularization. 
 We point out however that the analysis in \cite{Delosreyes2009} does not require the restrictive Assumptions \ref{assu:active} 
 and \ref{assu:ycont} and in addition applies to more general VIs of the second kind.
\end{remark}

\section{An inexact trust-region algorithm}
In this section we propose an inexact trust-region algorithm for the solution of the finite-dimensional optimization problem:
\begin{align}
\min ~& J(y,u)\\
\text{subject to: }& \langle Ay,v-y \rangle +g|v|_1- g|y|_1 \geq \langle u, v-y\rangle,  \text{ for all } v \in \mathbb R^n,
\end{align}
with $g>0$. The main difficulty of the method consists in computing a descent direction along which the algorithm has to perform the next step.
In the case of an empty biactive set, the derivative information is given by \eqref{eq:etanull}-\eqref{eq:rableq}. From the latter, existence of an adjoint state can be proved and an adjoint calculus may be performed. 

Since the information so obtained does not necessarily correspond to an element of the subdifferential, in case of a non-empty biactive set, we apply a trust-region scheme to provide robust iterates. In this context the adjoint related gradient is considered as an inexact version of a descent direction. Since in the applications we focus on, the biactive set is either empty or very small, such an approach is justified from the numerical point of view.

Indeed, by assuming that the biactive set $$B=\{ i: y_i=0, |q_i|=1 \},$$ is empty, the solution operator is G\^ateaux differentiable and the directional derivative $\eta=S'(u)h$ corresponds to the solution of the following system of equations:
\begin{align*}
\eta_i=0 &\text{ for }i:y_i=0,\\
\sum_{j:y_j \not =0} A_{i,j} \eta_j =h_i &\text{ for }i:y_i \not =0.
\end{align*}

To simplify the description of the algorithm, 
we confine ourselves to a quadratic cost functional of the form $J(y,u) = 1/2\, \|y - z\|^2 + \alpha/2\, \|u\|^2$, where $\|\,.\,\|$ denotes the 
Euclidian norm and $z\in \R^n$ is a given desired state.
Considering the reduced cost functional $$j(u)= \frac{1}{2} \|S(u)-z\|^2 + \frac{\alpha}{2} \|u\|^2,$$ the directional derivative is given by $$j'(u)h= (S(u)-z, S'(u)h)+ \alpha (u,h)= \sum_i (y_i-z_i)\eta_i + \alpha \sum_i u_i h_i.$$

Let us recall that the inactive set is given by $\mathcal I :=\{ i \in \{1, \dots, n \}: y_i \not =0 \}$. By reordering the indices such that the active and inactive ones occur in consecutive order, and defining the adjoint state $p \in \mathbb R^n$ as the solution to the system:
$$\begin{pmatrix}
I &0\\ 0 & A_{\mathcal I}^T
\end{pmatrix} p = y-z,
$$
where $A_{\mathcal I}$ corresponds to the block of $A$ with indexes $i, j$ such that $y_i \not = 0,y_j \not = 0$, we obtain that
$$j'(u)h= \sum_{i \in \mathcal I}p_i h_i + \alpha \sum_i u_i h_i$$ or, equivalently, $j'(u)=\begin{cases}
\alpha u_i & \text{ if } i \not \in \mathcal I\\ p_i+ \alpha u_i & \text{ if } i \in \mathcal I.
\end{cases}
$

Before stating the trust-region algorithm, let us introduce some notation to be used. The quadratic model of the reduced cost function is given by $$q_k(s)=j(u_k)+ g_k^T s+ \frac{1}{2} s^T H_k s,$$
where $g_k=j'(u_k)$ and $H_k$ is a matrix with second order information, obtained with the BFGS method. The trust region radius is denoted by $\Delta_k$ and the actual and predicted reductions are given by
$$ared_k(s^k):=j(u_k)-j(u_k+s^k) \text{ and }pred_k(s^k)=j(u_k)-q_k(s^k), \text{ respectively.}$$ The quality indicator is computed by $$\rho_k(s^k)=\frac{ared_k(s^k)}{pred_k(s^k)}.$$

The resulting trust region algorithm (of dogleg type) is given through the following steps:
\paragraph{\textbf{Trust region algorithm}}
\begin{enumerate}
\item Choose the parameter values $0<\eta_1<\eta_2<1$, $0<\gamma_0<\gamma_1<1<\gamma_2$, $\Delta_{min}\geq0.$
\item Choose the initial iterate $x_0\in\mathbb{R}^n$ and the trust region radius $\Delta_0>0,\ \Delta_0\geq\Delta_{min}\geq0.$
\item Compute the Cauchy step $s_c^k=-t^* g_k,$ where 
$$t^*=\left\{
\begin{matrix}
\displaystyle\frac{\Delta_k}{||g_k||},\ \ \text{ if }\ g_k^\top H_kg_k\leq0\\
\\
\min\left(\displaystyle\frac{||g_k||^2}{g_k^\top H_kg_k},\displaystyle\frac{\Delta_k}{||g_k||}\right),\ \ \text{ if }\ g_k^\top H_k g_k>0
\end{matrix}
\right.$$
and the Newton step $s^k_n= -H_k^{-1} g_k$. If $s^k_n$ satisfies the fraction of Cauchy decrease: 
$$\exists \delta \in (0,1] \text{ and } \beta \geq 1 \text{ such that }\|s^k\| \leq \beta \Delta_k \text{ and }pred_k(s^k) \geq \delta ~pred_k(s_c^k).$$
then $s^k=s^k_n$, else $s^k=s^k_c$.
\item If $\varrho_k(s^k)>\eta_2$, then
\[
u_{k+1}=u_k+s_k,\ \ \ \Delta_{k+1}\in\left[\Delta_k,\gamma_2\Delta_k\right]
\]
Else if $\varrho_k(s^k)\in(\eta_1,\eta_2)$, then
\[
u_{k+1}=u_k+s_k,\ \ \ \Delta_{k+1}\in\left[\max(\Delta_{min},\gamma_1\Delta_k),\Delta_k\right]
\]
Else if $\varrho_k(s^k)\leq\eta_1$, then
\[
u_{k+1}=u_k,\ \ \ \Delta_{k+1}\in\left[\gamma_0\Delta_k,\gamma_1\Delta_k\right]
\]
Repeat until stopping criteria.
\end{enumerate}

\subsection{Example}
We consider as test example the following finite-dimensional optimization problem:
\begin{align}
\min ~& J(y,u)=\frac{1}{2} \|y-z\|^2 + \frac{\alpha}{2} \|u\|^2\\
\text{subject to: }& \langle Ay,v-y \rangle +g|v|_1- g|y|_1 \geq \langle u, v-y\rangle,  \text{ for all } v \in \mathbb R^n, \label{numerics VI}
\end{align}
where $A$ corresponds to the finite differences discretization matrix of the negative Laplace operator in the two dimensional domain $\Omega=]0,1[^2$, $z =10 \sin(5 x_1) \cos(4 x_2)$ stands for the desired state and $\alpha$ and $g$ are positive constants. It is expected that as $g$ becomes larger the solution becomes sparser.

For solving \eqref{numerics VI} within the trust region algorithm a semismooth Newton method is used. The method is built upon a huberization of the $l_1$ norm and the use of dual information. Specifically,
we consider the solution of the regularized inequality:
\begin{align}
& Ay +q=u\\
& q-h_\gamma(y)=0,
\end{align}
where $\left( h_\gamma(y) \right)_i= g\frac{\gamma y_i}{\max(g, \gamma |y_i|)}$. Considering a generalized derivative of the max function, the following system has to be solved in each semismooth Newton iteration:
\begin{align}
& A \delta_y +\delta_q=u-Ay -q\\
& \delta_q-\frac{\gamma \delta_y}{\max(g,\gamma|y|)}+ \diag(\chi_{\mathcal I_\gamma}) \frac{\gamma^2 y^T \delta_y}{\max(g,\gamma|y|)^2} \frac{y}{|y|}=-q+h_\gamma(y), \label{eq:ssn2}
\end{align}
where $\left( \chi_{\mathcal I_\gamma} \right)_i:=\begin{cases}1 &\text{ if }\gamma |y_i| \geq g,\\ 0 &\text{ if not.} \end{cases}$, $\max(g,\gamma |y|):=\left( \max(g,\gamma |y_1|),\dots,\max(g,\gamma |y_n|) \right)^T$ and the division is to be understood componentwise. By using dual information in the iteration matrix (as in \cite{HintermuellerStadler2007}, \cite{dlRe2010}) the following modified version of \eqref{eq:ssn2} is obtained:
\begin{equation}
\delta_q-\frac{\gamma \delta_y}{\max(g,\gamma|y|)}+ \diag(\chi_{\mathcal I_\gamma}) \frac{\gamma^2 y^T \delta_y}{\max(g,\gamma|y|)^2} \frac{q}{\max(g,|q|)}=-q+h_\gamma(y). \label{eq:ssn2.1}
\end{equation}
This leads to a globally convergent iterative algorithm, which converges locally with superlinear rate.

The used trust region parameter values are $\eta_1=0.25$, $\eta_2=0.75$, $\gamma_1=0.5$, $\gamma_2=1.5$ and $\beta=1$. For the parameter values $\alpha=0.0001$ and $g=15$, and the mesh size step $h=1/80$, the algorithm requires a total number of 35 iterations to converge, for a stopping criteria given by $\|u_{k+1}-u_k\| \leq 1e-4$. The optimized state is shown in Figure \ref{Bsp1E}, where a large zone where the state takes value zero can be observed.
\begin{figure}[ht!]
\centering
\includegraphics[height=6cm]{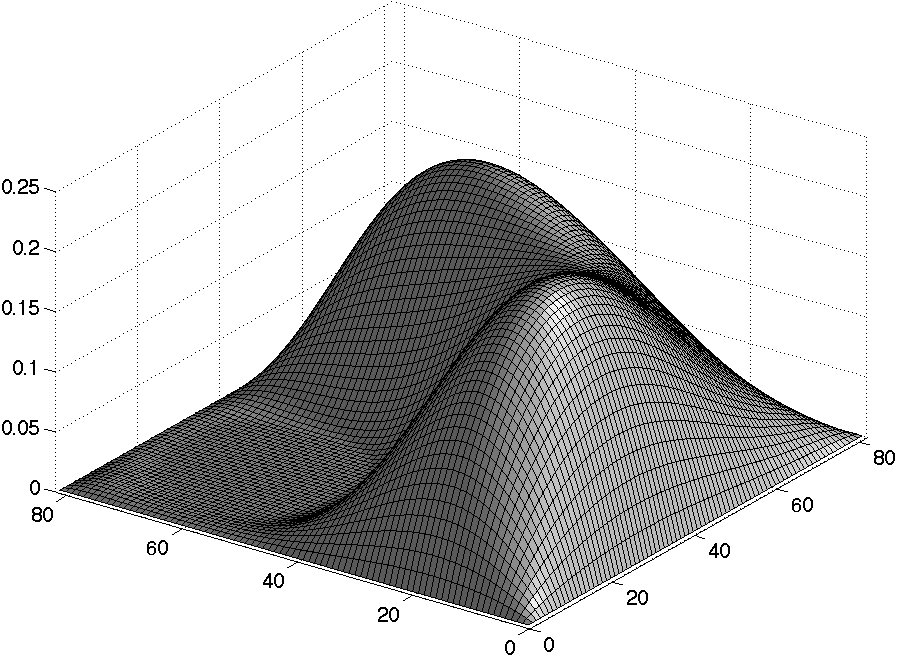}
\caption{Optimized state: on the left corner the sparse structure of the solution can be observed.}\label{Bsp1E}
\end{figure}

The algorithm was also tested for other values of the parameters $\alpha$ and $g$, yielding the convergence behaviour registered in Table \ref{table1}. Although the considered derivative information was inexact, the trust-region approach yields convergence in a relatively small number of iterations.
\begin{table}
\centering
\begin{tabular}{|l|r|r|r|r|}
\backslashbox{$\alpha$}{$g$}	&1	&5	&10	&15\\ \hline
0,1	&20	&28	&53	&- \\
0,01	 &23 	&24	&27	&32\\
0,001	&33	&48	&54	&31\\
0,0001	&69	&70	&62	&34
\end{tabular} \caption{Number of trust-region iterations for different $\alpha$ and $g$ values. Mesh size step $h=1/40$.}\label{table1}
\end{table} 

Further descent type directions to be used in the context of the trust-region methodology, as well as the convergence theory of the combined approach, will be investigated in future work.

\begin{appendix}

\section{Directional derivative of the $L^1$-norm}\label{sec:l1deriv}

\begin{proofof}{Lemma \ref{lem:l1deriv}}
 The definition of $\AA$ and $\absop'$ imply
 \begin{equation}\label{eq:l1deriv1}
 \begin{aligned}
  & \Big|\int_\Omega \Big(\frac{|y_n| - |y|}{t_n} - \absop'(y;\eta)\Big) \varphi\, dx\Big|\\
  &\quad \leq \Big|\int_\Omega \frac{|y_n| - |y + t_n\eta|}{t_n} \, \varphi\, dx\Big|
  + \underbrace{\Big|\int_{\AA} \Big(\frac{|y + t_n\eta| - |y|}{t_n} - |\eta|\Big) \varphi\, dx\Big|}_{\displaystyle{= 0}}\\
  &\qquad + \Big|\int_{\Omega\setminus\AA} \Big(\frac{|y + t_n \eta| - |y|}{t_n} - \sign(y)\eta\Big) \varphi\, dx\Big|.
 \end{aligned}
 \end{equation}
 By the compact embedding $H^1(\Omega)\embed\embed L^2(\Omega)$ we have 
 \begin{equation*}
  \frac{y_n - y}{t_n} \to \eta \text{ in } L^2(\Omega)
 \end{equation*}
 and thus
 \begin{equation}\label{eq:l1deriv2}
 \begin{aligned}
  \Big|\int_\Omega \frac{|y_n| - |y + t_n\eta|}{t_n} \, \varphi\, dx\Big|
  &\leq \int_\Omega \Big|\frac{y_n - y}{t_n} + \eta\Big| \, |\varphi|\, dx\\
  &\leq \Big\|\frac{y_n - y}{t_n} + \eta\Big\|_{L^1(\Omega)} \, \|\varphi\|_{L^\infty(\Omega)} \to 0.
 \end{aligned}
 \end{equation}
 Let $x \in \Omega\setminus\AA$ be an arbitrary common Lebesgue point of $y$ and $\eta$. 
 Then the directional differentiability of $\R \ni r \mapsto |r| \in \R$ yields
 \begin{equation*}
  \frac{|y(x) + t_n \eta(x)| - |y(x)|}{t_n} - \sign\big(y(x)\big)\eta(x) \to 0,
 \end{equation*}
 and, since almost all points in $\Omega$ are common Lebesgue points of $y$ and $\eta$, this pointwise convergence 
 holds almost everywhere in $\Omega\setminus\AA$. 
 Due to
 \begin{equation*}
  - 2 |\eta(x)|\leq \frac{|y(x) + t_n \eta(x)| - |y(x)|}{t_n} - \sign\big(y(x)\big)\eta(x) 
  \leq 2 |\eta(x)| \quad \text{a.e.\ in }\Omega,
 \end{equation*}
 Lebesgue dominated convergence theorem thus gives 
 \begin{equation*}
  \frac{|y + t_n \eta| - |y|}{t_n} - \sign(y)\eta \to 0 \text{ in } L^1(\Omega\setminus\AA).
 \end{equation*}  
 Therefore, we arrive at
 \begin{multline}\label{eq:l1deriv3}
  \Big|\int_{\Omega\setminus\AA} \Big(\frac{|y + t_n \eta| - |y|}{t_n} - \sign(y)\eta\Big) \varphi\, dx\Big|\\
  \leq \Big\|\frac{|y + t_n \eta| - |y|}{t_n} - \sign(y)\eta\Big\|_{L^1(\Omega\setminus\AA)} \, \|\varphi\|_{L^\infty(\Omega)} 
  \to 0.
 \end{multline}
 Inserting \eqref{eq:l1deriv2} and \eqref{eq:l1deriv3} in \eqref{eq:l1deriv1} yields the assertion.
\end{proofof}

\section{Boundedness for functions in $H^1(\Omega)$}\label{sec:stam}

For convenience of the reader, we prove Lemma \ref{lem:stam}. The arguments are classical and go back to \cite{Stampacchia/Kinderlehrer}.

\begin{proofof}{Lemma \ref{lem:stam}}
The truncated function defined in \eqref{eq:truncfunc} is equivalent to
\begin{equation*}
 w_k(x) = w(x) - \min\big((\max(w(x),-k),k\big)
\end{equation*}
and therefore \cite[Theorem A.1]{Stampacchia/Kinderlehrer} implies $w_k \in V$.

It remains to verify the $L^\infty$-bound in \eqref{eq:inftybound}. 
If $d=1$, then the assertion follows directly from \eqref{eq:stamest} and the Sobolev embedding $H^1(\Omega) \embed L^\infty(\Omega)$.

So assume that $d\geq 2$.
Then let $k \geq 0$ be given and set $A(k) :=\{x\in \Omega\;|\;|w(x)|\geq k\}$. Note that $w_k(x) = 0$ a.e.\ in $\Omega\setminus A(k)$.
Next let $h \geq k$ be arbitrary so that $w(x) \geq h \geq k$ a.e.\ in $A(h)$. 
Then Sobolev embeddings give that
\begin{equation}\label{eq:stam1}
\begin{aligned}
 \|w_k\|_{H^1(\Omega)}^2 \geq c\, \|w_k\|_{L^m(\Omega)}^2 &= c\Big(\int_{A(k)} \big| |w| - k \big|^m dx\Big)^{2/m}\\
 &\geq c\int_{A(h)}(h-k)^m dx^{2/m} =  c\,(h-k)^2 |A(h)|^{2/m},
\end{aligned}
\end{equation}
where $m = 2d/(d-2)$, see e.g.\ ... On the other hand, \eqref{eq:stamest} implies
\begin{equation*}
 \alpha\,\|w_k\|^2 \leq \int_{A(k)} f\,w_k\,dx \leq \|f\|_{L^{m'}(A(k))} \, \|w_k\|_{L^m(A(k))} 
 \leq c\,\|f\|_{L^{m'}(A(k))} \, \|w_k\|_{H^1(\Omega)},
\end{equation*}
where $m'$ is the conjugate exponent to $m$, i.e.\ $1/m + 1/m' = 1$. Note that 
\begin{equation*}
 m' = \frac{m}{m-1} = \frac{d}{d/2 + 1} \leq \frac{d}{2} < p, \quad \text{if } d \geq 2,
\end{equation*}
and thus $f\in L^{m'}(\Omega)$ by the assumption on $f$ in Lemma \ref{lem:stam}. 
Together with Young's inequality, then H\"older's inequality yields
\begin{equation}\label{eq:stam2}
 \|w_k\|^2 \leq c\Big(\int_{A(k)} |f|^{m'}\,dx\Big)^{2/m'} \leq c\,\|f\|_{L^p(\Omega)}^2\,|A(k)|^{2r/m'}
\end{equation}
with $r = p/(p-m') \geq 1$ so that $r' = r/(r-1) = p/m'$. By setting
\begin{equation}\label{eq:expo}
 s = \frac{m}{m'}\, r = \frac{p}{(m'-1)(p-m')}
\end{equation}
we infer from \eqref{eq:stam1} and \eqref{eq:stam2} that 
\begin{equation}\label{eq:stam3}
 |A(h)|^{2/m} \leq c\,\|f\|_{L^p(\Omega)}^2\,\frac{1}{(h-k)^2} \, \big(|A(h)|^{2/m}\big)^s\quad  \text{for all } h > k \geq 0.
\end{equation}
Since $m > 2$, we have $m'< 2$ and therefore $(m'-1)(p-m') < p-m' < p$ such that \eqref{eq:expo} gives in turn $s>1$. 
In this case, according to \cite[Lemma B.1]{Stampacchia/Kinderlehrer}, it follows from \eqref{eq:stam3} that the nonnegative and non-increasing 
function $\R\ni h \mapsto |A(h)|^{2/m} \in \R$ admits a zero at 
\begin{equation*}
 h^* = 2^{s/(s-1)} \sqrt{c \,|\Omega|^{2(s-1)/m}}\, \|f\|_{L^p(\Omega)}.
\end{equation*}
By definition, $|A(h^*)| = 0$ is equivalent to $|w(x)| \leq h^*$ a.e.\ in $\Omega$, which yields the assertion.
\end{proofof}

\end{appendix}

\section*{Acknowledgement}

The authors would like to thank Gerd Wachsmuth (TU Chemnitz) for his hint concerning strong stationarity.

This work was supported by a DFG grant within the Collaborative Research Center SFB 708 
(\emph{3D-Surface Engineering of Tools for Sheet Metal Forming – Manufacturing, Modeling, Machining}), 
which is gratefully acknowledged.

\bibliographystyle{plain}
\bibliography{biblio}
\end{document}